\documentclass[12pt,reqno]{amsart}
\usepackage{amssymb}
\usepackage{graphicx}
\usepackage{bm}
\usepackage{hyperref}
\usepackage{enumerate}
\allowdisplaybreaks[1]

\newcommand{\B}{\mathbb B}
\newcommand{\C}{\mathbb C}
\newcommand{\D}{\mathbb D}

\newcommand{\N}{\mathbb N}
\newcommand{\R}{\mathbb R}

\renewcommand{\H}{\mathbb H}
\renewcommand{\P}{\mathbb P}
\newcommand{\Q}{\mathbb Q}

\newcommand{\Z}{\mathbb Z}

\newcommand{\cE}{E}

\newcommand{\cF}{\mathcal F}

\newcommand{\re}{\mathrm{Re}}
\newcommand{\im}{\mathrm{Im}}

\renewcommand{\a}{\alpha}
\renewcommand{\b}{\beta}
\newcommand{\g}{\gamma}

\renewcommand{\l}{\lambda}

\newcommand{\h}{\vartheta}

\newcommand{\f}{\varphi}

\newcommand{\G}{\mathit{\Gamma}}

\newcommand{\w}{\omega}

\newcommand{\SL}{\mathrm{SL}}
\newcommand{\PSL}{\mathrm{PSL}}
\newcommand{\GL}{\mathrm{GL}}

\newcommand{\PGa}{\mathrm{P}\Gamma}

\renewcommand{\i}{\mathbf{i}}

\newcommand{\la}{\langle}
\newcommand{\ra}{\rangle}
\newcommand{\tr}{\;^t}

\newcommand{\pr}{\mathrm{pr}}

\newcommand{\comment}[1]{}
\newcommand{\wt}[1]{\widetilde{#1}}
\newcommand{\wh}[1]{\widehat{#1}}

\newcommand{\ds}[1]{\displaystyle{#1}}

\title[Analogies of Jacobi's formula]
{Analogies of Jacobi's formula}
\author{Keiji Matsumoto}
 \address[Matsumoto]{
 Department of Mathematics,
 Faculty of Science
 Hokkaido University,
 Sapporo 060-0810, Japan
}

\date{today}

\theoremstyle{plain} 
\newtheorem{theorem}{\indent\sc Theorem}[section]
\newtheorem{lemma}[theorem]{\indent\sc Lemma}
\newtheorem{cor}[theorem]{\indent\sc Corollary}
\newtheorem{proposition}[theorem]{\indent\sc Proposition}

\theoremstyle{definition} 
\newtheorem{definition}[theorem]{\indent\sc Definition}
\newtheorem{fact}[theorem]{\indent\sc Fact}
\newtheorem{remark}[theorem]{\indent\sc Remark}

\numberwithin{equation}{section}

\begin{document}

\begin{abstract}
By considering Schwarz's map for 
the hypergeometric differential equation with parameters 
$(a,b,c)=(1/6,1/2,1)$ or $(1/12,5/12,1)$,
we give some analogies of Jacobi's formula
$\h_{00}(\tau)^2= F(1/2,1/2,1;\l(\tau))$, 
where $\h_{00}(\tau)$ and $\l(\tau)$ are 
the theta constant and the lambda function 
defined on the upper-half plane,  
and 
$F(a,b,c;z)$ is the hypergeometric series defined on the unit disk.
As applications of our formulas, we give several functional equations 
for $F(a,b,c;z)$.  
\end{abstract}

\maketitle
\section{Introduction}

Schwarz's map is defined by the analytic continuation of 
the ratio of linearly independent local solutions to 
the hypergeometric differential equation $\cE(a,b,c)$ given in \eqref{eq:HGDE}. 
It is a multi-valued map 
from $Z=\C-\{0,1\}$ to the complex projective line $\P^1$ in general. 
Schwarz's map $\f_0$ for $(a,b,c)=(\frac{1}{2},\frac{1}{2},1)$ is 
given by the map
\begin{equation}
\label{eq:intr-f0}
z\mapsto \f_0(z)=\i \cdot 
\frac{F(\frac{1}{2},\frac{1}{2},1;1-z)}{F(\frac{1}{2},\frac{1}{2},1;z)},
\end{equation} 
around $z=\frac{1}{2}$, where $\i=\sqrt{-1}$, and 
$F(a,b,c;z)$ is the hypergeometric series given in \eqref{eq:HGS}. 
It is known that the image  $\f_0(Z)$ is the upper-half plane $\H$, and that 
the projective monodromy of $\f_0$ is $\PGa(2)=\Gamma(2)/\{\pm I_2\}$, 
where $I_2$ is the unit matrix of size $2$, and  
$\Gamma(2)$ is the principally congruence subgroup of $\SL_2(\Z)$ 
of level $2$. 
The inverse of $\f_0$ is expressed by the lambda function 
\begin{equation}
\label{eq:intr-lambda}
\tau\mapsto \lambda(\tau)=\frac{\h_{10}(\tau)^4}{\h_{00}(\tau)^4},
\end{equation}
where $\h_{pq}(\tau)$ is the theta constant with characteristics $p,q$
given in  \eqref{eq:theta}. 
Jacobi's formula is
$$\h_{00}(\tau)^2=F(\frac{1}{2},\frac{1}{2},1;\lambda(\tau))
$$
where $\tau$ is a complex number with large imaginary part.  
Here note that this formula is an equality between 
the denominators in \eqref{eq:intr-f0} and \eqref{eq:intr-lambda} 
under the relation 
$z=\l(\tau)$ in \eqref{eq:intr-lambda}.
Some analogies of Jacobi's formula are given in 
\cite[Theorem 2.3, Theorem 2.6]{BB2},
which relate $F(\frac{1}{3},\frac{2}{3},1;z)$ and 
$F(\frac{1}{4},\frac{3}{4},1;z)$ to modular forms 
with respect to their projective monodromy group 
isomorphic to 
the triangle groups $(3,\infty,\infty)$ and $(2,\infty,\infty)$, 
respectively.

In this paper,  we give some analogies of Jacobi's formula by considering 
Schwarz's maps $\f_1$ for $(a,b,c)=(\frac{1}{6},\frac{1}{2},1)$ 
and $\f_2$ for $(a,b,c)=(\frac{1}{12},\frac{5}{12},1)$. 
There exists a basis of the space of local solutions to 
$\cE(\frac{1}{6},\frac{1}{2},1)$ (or $\cE(\frac{1}{12},\frac{5}{12},1)$)
around $z=\frac{1}{2}$ such 
that the image of $\f_1$ (or $\f_2$) is an open dense subset in $\H$
and that the projective monodromy of $\f_1$  (or $\f_2$) is a subgroup of 
$\PSL_2(\Z)=\SL_2(\Z)/\{\pm I_2\}$   
isomorphic to the triangle group $(3,3,\infty)$ (or $(2,3,\infty)$). 
We give each basis 
as linear combinations of $F(a,b,a+b-c+1;1-z)$ and $F(a,b,c;z)$ 
for $(a,b,c)=(\frac{1}{6},\frac{1}{2},1)$, $(\frac{1}{12},\frac{5}{12},1)$,
and the circuit matrices along $\rho_0$ and $\rho_1$ 
with respect to them as 
$$\begin{array}{lcl}
T^2,\ \w^2 W
&\textrm{for} & (a,b,c)=(\frac{1}{6},\frac{1}{2},1),\\
T,\quad \i J
&\textrm{for} & (a,b,c)=(\frac{1}{12},\frac{5}{12},1),\\
  \end{array}
$$
where $\rho_k$ $(k=0,1)$ is a positively oriented circle with
terminal $z=\frac{1}{2}$, radius $\frac{1}{2}$ and center $z=k$, and 
$$\w=\dfrac{-1+\sqrt3\i}{2}, \quad 
T=\begin{pmatrix}
 1 & 1 \\
 0 & 1
 \end{pmatrix},\quad
W=\begin{pmatrix}
-1 &-1 \\
 1 & 0
 \end{pmatrix},\quad
J=\begin{pmatrix}
 0 & 1 \\
-1 & 0
 \end{pmatrix}. 
$$
The group generated by $T$ and $J$ is $\SL_2(\Z)$, and that by $T^2$ and $W$ is 
$$\Gamma(2)^{1/3}=\{g\in \SL_2(\Z)\mid g^3\in \Gamma(2)\},$$
which is an intermediate group of $\SL_2(\Z)$ and $\Gamma(2)$.
Note that the quotients $\SL_2(\Z)/\Gamma(2)$ and $\Gamma(2)^{1/3}/\Gamma(2)$ 
are isomorphic to the symmetric group $S_3$ of degree $3$ and 
the alternating group $A_3$ of degree $3$, respectively.
We express the inverses of $\f_1$ and $\f_2$ by the lambda function; in fact, 
$\f_1^{-1}$ is given in Theorem \ref{th:inv-Sch} 
by 
an invariant 
$$\nu(\tau)=\frac{3\sqrt{3}\i\l(\tau)(1-\l(\tau))}{(\l(\tau)+\w^2)^3}$$
under the action of $\Gamma(2)^{1/3}/\Gamma(2)\simeq A_3$, and 
$\f_2^{-1}$ is given in Theorem \ref{th:inv-f2} 
by the reciprocal of the $j$-invariant
$$j(\tau)=\frac{4}{27} \frac{(\l(\tau)^2-\l(\tau)+1)^3}
{\l(\tau)^2(1-\l(\tau))^2} . 
$$
By using the denominators of $\f_1,\f_1^{-1}$ and $\f_2,\f_2^{-1}$,
we have analogies of Jacobi' formula: 
\begin{align*}
\h_{00}(\tau)^4+\w\h_{10}(\tau)^4&=
F(\frac{1}{6},\frac{1}{2},1;\nu(\tau))^2,\\
%
%
E_4(\tau)
&=F(\frac{1}{12},\frac{5}{12},1;\frac{1}{j(\tau)})^4,
%
\end{align*}
in Theorems \ref{th:Ana1-Jacobi}, 
\ref{th:Ana2-Jacobi},  
where $E_4(\tau)$ is the normalized Eisenstein series of weight $4$.  
We remark that the factor $\w^2$ in the circuit matrix $\w^2W$ 
along $\rho_1$ plays an important role in our proof of the former equality. 

By eliminating $\h_{pq}(\tau)$ from Jacobi's formula and its analogies, 
we have functional equations for $F(a,b,c;z)$ such as 
$$
F(\frac{1}{12},\frac{5}{12},1;\frac{z^2}{4(z-1)})^2
=F(\frac{1}{6},\frac{1}{2},1;z)F(\frac{1}{6},\frac{1}{2},1;\frac{z}{z-1}); 
$$
see Corollary \ref{eq:FE-HGS} for other formulas. 


We summarize the contents of this paper.
In \S2, we collect facts for the hypergeometric series and 
the hypergeometric differential equations. 
We cannot select a well-used pair of 
$z^{1-c}F(a-c+1,b-c+1,2-c;z)$ and $F(a,b,c;z)$
as a basis of the space of local solutions to $\cE(a,b,c)$ around 
$\dot z=\frac{1}{2}$, since we set $c=1$ in our study.  
We show the linear independence of $F(a,b,a+b-c+1;1-z)$ and $F(a,b,c;z)$
under some conditions on $a,b,c$, and give a simple way to obtain 
circuit matrices of 
$\rho_0$ and $\rho_1$ with respect to constant multiples of them. 
In \S3, we review properties of the theta constant by referring to \cite{Mu}.
We give actions of $\SL_2(\Z)/\Gamma(2)\simeq S_3$ on $\h_{pq}(\tau)^2$.
In \S4, we give facts about 
Schwarz's map $\f_0$ for $(a,b,c)=(\frac{1}{2},\frac{1}{2},1)$ 
by referring to \cite{Yo}. 
We study Schwarz's map $\f_1$ for $(a,b,c)=(\frac{1}{6},\frac{1}{2},1)$ 
in \S5 and Schwarz's map $\f_2$ for $(a,b,c)=(\frac{1}{12},\frac{5}{12},1)$ 
in \S6, and give functional equations for $F(a,b,c;z)$ in \S7. 
\section{Fundamental properties of $F(a,b,c;z)$ }
We begin with defining the hypergeometric series.
\begin{definition}
\label{def:HGS}
The hypergeometric series $F(a,b,c;z)$ is defined by 
\begin{equation}
\label{eq:HGS}
F(a,b,c;z)=\sum_{n=0}^\infty \frac{(a,n)(b,n)}{(c,n)(1,n)}z^n, 
\end{equation}
where $z$ is a complex variable,
$a,b,c$ are complex parameters with $c\ne 0,-1,-2,\cdots$, and 
$(a,n)=\G(a+n)/\G(a)=a(a+1)\cdots(a+n-1)$.
It converges absolutely and uniformly on any compact set 
in the unit disk $\B_0=\{z\in \C\mid |z|<1\}$ for any fixed $a,b,c$.
\end{definition}

The radius of the convergence of $F(a,b,c;z)$ is $1$ for generic $a,b,c$.
Gauss's identity gives the limit of $F(a,b,c;z)$ as 
$z\to 1$ with $|z|<1$.

\begin{fact}[Gauss's identity {\cite[(14) in p.61]{Er}}]
\label{fact:G-K}
If 
$\re(c-a-b)>0$ 
then 
\begin{equation}
\label{eq:G-K}
\lim_{z\to 1,|z|<1}F(a,b,c;z)=\frac{\G(c)\G(c-a-b)}{\G(c-a)\G(c-b)}.
\end{equation}
\end{fact}

The series $F(a,b,c;z)$ satisfies the hypergeometric differential equation 
\begin{equation}
\label{eq:HGDE}
\cE(a,b,c):\big[z(1-z)\frac{d^2}{dz^2}+\{c-(a+b+1)z\}
\frac{d}{dz}-ab\big]\cdot f(z)=0,
\end{equation} 
which is a second order linear ordinary differential equation with 
regular singular points $z=0,1,\infty$. 
The space $\mathcal{F}_{\dot z}(a,b,c)$ of 
(single-valued) local solutions to $\cE(a,b,c)$ around a point 
$\dot z\in Z=\C-\{0,1\}$ is a $2$-dimensional complex vector space.
We take $\dot z=\frac{1}{2}\in Z$, and fix it as a base point of $Z$. 
We give a basis of $\mathcal{F}_{\dot z}(a,b,c)$ for $\dot z=\frac{1}{2}$. 

\begin{lemma}[A fundamental system of solutions to $\cE(a,b,c)$]
\label{lem:basis}
Suppose that $a,b,c$ satisfy $a,b\notin \Z$ and 
$\re(a),\re(b-c+1),\re(c-a)>0$. 
\begin{enumerate}
\item \label{enum:Basis}
The integrals 
\begin{equation}
\label{eq:Basis}
\begin{array}{l}
\ds{f_1(z)=\int_{-\infty}^{0} e^{-\pi\i(a-1)}(-t)^{b-c}(z-t)^{-b}(1-t)^{c-a-1}dt,}
\\[3mm]
\ds{f_2(z)=\int_1^\infty t^{b-c}(t-z)^{-b}(t-1)^{c-a-1}dt,}
\end{array}
\end{equation}
converge and satisfy the hypergeometric differential equation $\cE(a,b,c)$,  
where we assign each branch of the integrands in \eqref{eq:Basis}
over the integral intervals to positive real values 
for real $z$ near to $\dot z =\frac{1}{2}$.
They admit expressions in terms of the hypergeometric series as 
\begin{equation}
\label{eq:Basis-HG}
\begin{array}{l}
f_1(z)=e^{-\pi\i(a-1)}B(a,b-c+1)F(a,b,a+b-c+1;1-z),\\[3mm] 
f_2(z)=B(a,c-a)F(a,b,c;z),
\end{array}
\end{equation}
where $B$ denotes the beta function. 

\item \label{enum:Wronskian}
The functions $f_1(z)$ and $f_2(z)$ span the space 
$\mathcal{F}_{\dot z}(a,b,c)$.   
In particular, we have 
\begin{equation}
\label{eq:Wronskian1}
\begin{array}{ll}
&
\det\begin{pmatrix}
F(a,b,a\!+\!b\!-\!c\!+\!1;1-z)& 
\dfrac{d}{dz}F(a,b,a\!+\!b\!-\!c\!+\!1;1-z)\\
F(a,b,c;z)& \dfrac{d}{dz}F(a,b,c;z)\\ 
\end{pmatrix}\\
=&\dfrac{\G(c)\G(a+b-c+1)}{\G(a)\G(b)}z^{-c}(1-z)^{c-a-b-1}.
\end{array}
\end{equation}
\end{enumerate}
\end{lemma}
\begin{proof}
\eqref{enum:Basis} 
It is elementary to show the convergence of these integrals, 
which are given by a variable change $s=1/t$ to usual Euler type integrals 
$$\int_\ell s^{a-1}(1-s)^{c-a-1}(1-sz)^{-b}ds$$
satisfying $\cE(a,b,c)$, 
where $\ell$ is the open interval $(0,1)$ or $(1,\infty)$. We remark that  
the integrand 
$e^{-\pi\i(a-1)}(-t)^{b-c}(z-t)^{-b}(1-t)^{c-a-1}$ 
in $f_1(z)$ is the analytic continuation of 
$u(t)=t^{b-c}(t-z)^{-b}(t-1)^{c-a-1}$ via the upper-half plane of the $t$-space. 
The expression of $f_2(z)$ in \eqref{eq:Basis-HG} is a direct consequence
of the above, and we obtain that of $f_1(z)$ in \eqref{eq:Basis-HG}
by a further variable change $t'=1-t$ to the integral expressing $f_1(z)$. 
Note that 
neither the parameter $c$ nor $a+b-c+1$ in \eqref{eq:Basis-HG}
becomes a non-positive integer
under the conditions $\re(a)$, $\re(b-c+1)$, $\re(c-a)>0$. 

\medskip\noindent
\eqref{enum:Wronskian} By \cite[Theorem 1.1]{Va}, we have  
$$
\det\begin{pmatrix}
f_1(z) & \dfrac{d}{dz} f_1(z)\\
f_2(z) & \dfrac{d}{dz} f_2(z)\\
\end{pmatrix}
=e^{-\pi\i(a-1)}\dfrac{\G(a)\G(c-a)\G(b-c+1)}{\G(b)}z^{-c}(1-z)^{c-a-b-1}.$$
Since the last term never vanishes under the assumption,
$f_1(z)$ and $f_2(z)$ are linearly independent, 
and span the space $\mathcal{F}_{\dot z}(a,b,c)$. 
We can rewrite this equality to \eqref{eq:Wronskian1}
by \eqref{eq:Basis-HG}. 
\end{proof}

On a small disk with a point hole in the center $z_0$ $(z_0=0,1,\infty)$, 
there are (multi-valued) solutions to $\cE(a,b,c)$ of the forms 
$$z^{e_{0,i}}f_{0,i}(z),
\quad 
(1-z)^{e_{1,i}}f_{1,i}(z),
\quad 
(1/z)^{e_{\infty,i}}f_{\infty,i}(z),
$$
where $f_{z_0,i}(z)$ are holomorphic functions around $z=z_0$ with 
$f_{z_0,i}(z_0)=1$ ($z_0=0,1,\infty$), and $e_{z_0,i}$ are given in 
Table \ref{tab:R-scheme} called Riemann's scheme. 
\begin{table}[hbt]
$$
\begin{array}{|c|ccc|}
\hline
z_0 & 0 & 1 & \infty \\
\hline
e_{z_0,1}&  0  &   0   & a\\
e_{z_0,2}& 1-c & c-a-b & b\\
\hline
e_{z_0,2}-e_{z_0,1}&1-c & c-a-b & b-a\\
\hline
  \end{array}
$$
\caption{Riemann's scheme}
\label{tab:R-scheme}
\end{table}

Any element in $\mathcal{F}_{\dot z}(a,b,c)$ admits the analytic continuation 
along any path in $Z$. 
In particular, a loop $\rho$ in $Z$ with the base point $\dot z$ leads to 
a linear transformation of $\mathcal{F}_{\dot z}(a,b,c)$ called the circuit 
transformation along $\rho$. We have 
a homomorphism from the fundamental group $\pi_1(Z,\dot z)$ to 
the general linear group $\mathrm{GL}(\mathcal{F}_{\dot z}(a,b,c))$ of 
$\mathcal{F}_{\dot z}(a,b,c)$, which is called the monodromy representation of 
$\cE(a,b,c)$.
We take loops $\rho_0$ and $\rho_1$ with terminal $\dot z$ as 
positively oriented circles with radius $\frac{1}{2}$ and 
center $0$ and $1$, respectively. 
Since $\pi_1(Z,\dot z)$ is a free group generated by $\rho_0$ and $\rho_1$, 
the monodromy representation of $\cE(a,b,c)$ is uniquely characterized 
by the images of $\rho_0$ and $\rho_1$.
We give their explicit forms 
with respect to the basis 
of $\mathcal{F}_{\dot z}(a,b,c)$ in Lemma \ref{lem:basis}.

\begin{lemma} 
\label{lem:monodromy}
The loops $\rho_0$ and $\rho_1$  
lead to linear transformations sending the 
basis $\mathbf{F}(z)=\tr(f_1(z),f_2(z))$ of $\mathcal{F}_{\dot z}(a,b,c)$ 
to $M_0 \mathbf{F}(z)$ and $M_1\mathbf{F}(z)$, where the circuit matrices 
$M_0$ and $M_1$ are 
\begin{equation}
\label{eq:Cir-Mat}
M_0=
\begin{pmatrix}
e^{-2\pi\i c} & e^{-2\pi\i c}-e^{-2\pi\i a}  \\  
0 & 1
\end{pmatrix},
\quad 
M_1=
\begin{pmatrix}
1   & 0\\
e^{2\pi\i(c- b)}-1 & e^{2\pi\i(c-a-b)}
\end{pmatrix}.
\end{equation}
\end{lemma}

\begin{proof}
Since $f_2(z)$ and $f_1(z)$ are single valued 
on $\B_0=\{z\in \C\mid |z|<1\}$ and $\B_1=\{z\in \C\mid |1-z|<1\}$, 
respectively, the circuit matrices $M_0$ and $M_1$ take forms of 
$$M_0=
\begin{pmatrix}
y_0'& y_0  \\  
0 & 1
\end{pmatrix},
\quad 
M_1=
\begin{pmatrix}
1   & 0\\
y_1   & y_1' 
\end{pmatrix}.
$$
Since the eigenvalues of $M_0$ are $1$ and $e^{-2\pi\i c}$ and 
those of $M_1$ are $1$ and $e^{2\pi\i(c-a-b)}$ by Riemann's scheme  
in Table \ref{tab:R-scheme}, we have 
$y_0'=e^{-2\pi\i c}$ and $y_1'=e^{2\pi\i(c-a-b)}$. 
We set 
$$M'_\infty
=(M_1M_0)^{-1}
=\begin{pmatrix}
y_0y_1e^{2\pi\i(a+b)}+e^{2\pi\i c}   & -y_0e^{2\pi\i(a+b)}\\
-y_1e^{2\pi\i(a+b-c)}   & e^{2\pi\i(a+b-c)}
\end{pmatrix},
$$
which is the circuit matrix of the loop 
$(\rho_1\rho_0)^{-1}=\rho_0^{-1}\rho_1^{-1}$ with respect to $\mathbf{F}(z)$.
Since $\rho_0^{-1}\rho_1^{-1}$ is a loop turning $z=\infty$ 
once around positively, the eigenvalues of $M'_\infty$ are  
$e^{2\pi\i a}$, $e^{2\pi\i b}$ by  Riemann's scheme in Table \ref{tab:R-scheme}. 
Thus we have
$$\mathrm{tr}(M'_\infty)=
y_0y_1e^{2\pi\i (a+b)}+e^{2\pi\i c}+e^{2\pi\i(a+b-c)}=e^{2\pi\i a}+e^{2\pi\i b},
$$
which relates $y_0$ and $y_1$ as 
$$
y_1=-\frac{(e^{2\pi\i a}-e^{2\pi\i c})(e^{2\pi\i b}-e^{2\pi\i c})}
{y_0e^{2\pi\i (a+b+c)}},
$$
where $\mathrm{tr}(M'_\infty)$ denotes the trace of $M'_\infty$. 
By eliminating $y_1$ in the expression of $M'_\infty$, we obtain 
an $e^{2\pi\i b}$-eigenvector 
$$(e^{2\pi\i a}-e^{2\pi\i c}, e^{2\pi\i(a+c)} y_0)$$
of $M'_\infty$. 

On the other hand, since $\rho_0^{-1}\rho_1^{-1}$ is
the loop connecting the end of $\rho_0^{-1}$ to the start 
of $\rho_1^{-1}$, it is homotopic to a loop $\rho'_\infty$ 
starting from $\dot z$ approaching $\infty$ via the lower-half plane 
of $Z$, turning $z=\infty$ once around positively, and 
tracing back to $\dot z$, see Figure \ref{fig:path}.
Note the path $\ell_{01}$ from $t=0$ to 
$t=1$ via the upper-half plane of the $t$-space is not involved 
in the movement of $z$ along $\rho'_\infty$, and that 
$\arg(t-z)$ on $\ell_{01}$ decreases $2\pi$ 
by the continuation along $\rho'_\infty$. 
Thus $\ds{\int_{\ell_{01}} u(t)dt}$ 
is an $e^{2\pi\i b}$-eigen function under the analytic continuation 
along $\rho'_\infty$, where the branch of $u(t)$ on $\ell_{01}$ is assigned 
by the continuation of that on the open interval $(1,\infty)$ 
defining $f_2(z)$ via the upper-half plane of the $t$-space. 
By Cauchy's integral theorem,  
this integral is expressed as a linear combination 
$$-f_1(z) -f_2(z)=-(1,1)\mathbf{F}(z).$$

Since each of $(e^{2\pi\i a}-e^{2\pi\i c}, e^{2\pi\i(a+c)} y_0)$ 
and $(1,1)$ is an $e^{2\pi\i b}$-eigenvector of $M'_\infty$, 
they are parallel. Hence we have 
$$e^{2\pi\i (a+c)} y_0=e^{2\pi\i a}-e^{2\pi\i c}
\Leftrightarrow y_0=e^{-2\pi\i c}-e^{-2\pi\i a}. 
$$
This expression of $y_0$ and the relation between $y_0$ and $y_1$ yield 
\eqref{eq:Cir-Mat}.
\begin{figure}[hbt]
\includegraphics[width=12cm]{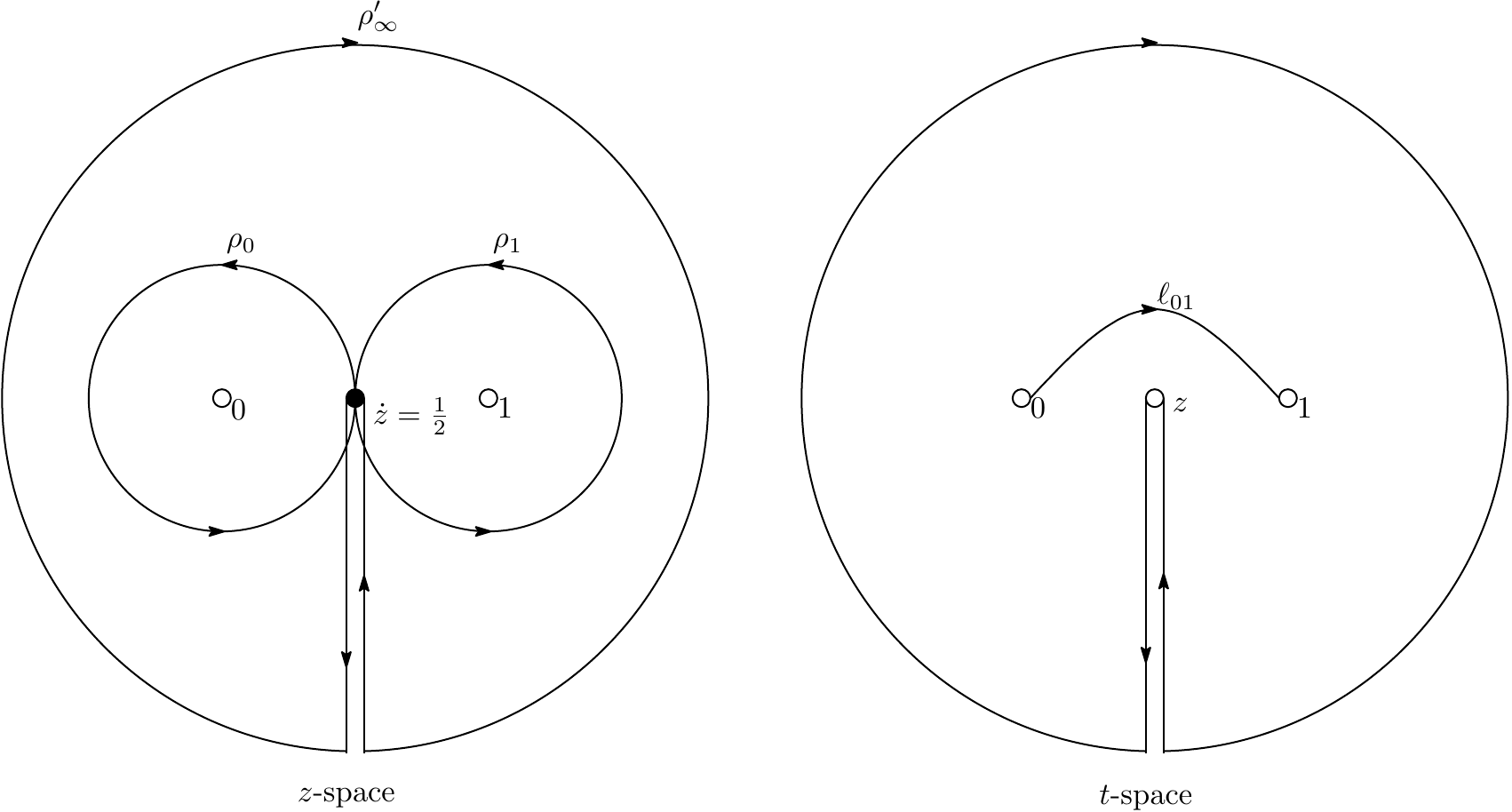} 
\caption{Loops in $Z$ and the path $\ell_{01}$}
\label{fig:path} 
\end{figure}
\end{proof}

\begin{remark}
\label{rem:eigenvects}
The integrals 
$$\int_0^z u(t)dt,\quad \int_z^1 u(t)dt$$
are an eigen function of the circuit transformation along $\rho_0$ 
with eigenvalue $e^{-2\pi\i c}$ and that along $\rho_1$ 
with eigenvalue $e^{2\pi\i(c-a-b)}$, respectively, where 
branches of $u(t)$ on $(0,z)$ and $(z,1)$ are assigned as the 
analytic continuations of that on $(0,1)$ via the upper-half space of 
the $t$-space. 
These are expressed by linear combinations of $f_1(z)$ and $f_2(z)$ as 
\begin{align*}
\int_0^z u(t)dt&=-\frac{e^{2\pi\i c}-1}{e^{2\pi\i b}-1}f_1(z)
+\frac{e^{2\pi\i a}-e^{2\pi\i c}}{e^{2\pi\i(a+b)}-e^{2\pi\i a}}f_2(z)
,\\
\int_z^1 u(t)dt&=\frac{e^{2\pi\i c}-e^{2\pi\i b}}{e^{2\pi\i b}-1}f_1(z)
-\frac{e^{2\pi\i (a+b)}-e^{2\pi\i c}}{e^{2\pi\i (a+b)}-e^{2\pi\i a}}f_2(z)
.
\end{align*}
\end{remark}

\begin{remark}
\label{rem:rho_infty}
In Proof of Lemma \ref{lem:monodromy}, 
we consider the loop $\rho'_\infty$ homotopic to $(\rho_1\rho_0)^{-1}$
and its circuit matrix 
$$M'_{\infty}=(M_1M_0)^{-1}=\begin{pmatrix}
e^{2\pi\i a}+e^{2\pi\i b}-e^{2\pi\i(a+b-c)} & e^{2\pi\i b}-e^{2\pi\i(a+b-c)}\\
e^{2\pi\i(a+b-c)}-e^{2\pi\i a} &e^{2\pi\i(a+b-c)}
\end{pmatrix}$$
with respect to $\mathbf{F}(z)$. 
Its eigen row vectors of eigenvalues $e^{2\pi\i a}$ and $e^{2\pi\i b}$ are 
$$(e^{2\pi\i (a+b)}-e^{2\pi\i (a+c)}, e^{2\pi\i (a+b)}-e^{2\pi\i (b+c)}),\quad 
(1,1),  
$$
respectively.
We set $\rho_\infty=(\rho_0\rho_1)^{-1}$, which is homotopic to a loop  
starting from $\dot z$ approaching $\infty$ via the upper-half plane 
of $Z$, turning $z=\infty$ once around positively, 
and tracing back to $\dot z$.  
Its circuit matrix $M_\infty$ with respect to $\mathbf{F}(z)$ is 
$$M_{\infty}=(M_0M_1)^{-1}=\begin{pmatrix}
e^{2\pi\i c} & e^{2\pi\i (c-a)}-1\\
e^{2\pi\i(a+b)}-e^{2\pi\i(a+c)} &e^{2\pi\i a}+ e^{2\pi\i b}- e^{2\pi\i c}
\end{pmatrix}.$$
Its eigen row vectors of eigenvalues $e^{2\pi\i a}$ and $e^{2\pi\i b}$ are 
$$(e^{2\pi\i (a+b)}-e^{2\pi\i (a+c)}, e^{2\pi\i a}-e^{2\pi\i c}),\quad 
(e^{2\pi\i a},1),  
$$
respectively.
\end{remark}

\begin{remark}
\label{rem:discontinuous}
We can make the analytic continuation of the hypergeometric series $F(a,b,c;z)$ 
to the simply connected domain $\C-[1,\infty)$ as a solution to $\cE(a,b,c)$.
We use the same symbol $F(a,b,c;z)$ for this continuation, 
which is a single-valued holomorphic function on $\C-[1,\infty)$. 
\end{remark}

\begin{remark}
We set $\ds{B(a,c-a)F(a,b,c;z)=\int_1^\infty u(t)dt}$ in the second entry of 
the column vector $\mathbf{F}(z)$ for the compatibility of 
the monodromy representation with the action of linear fractional 
transformations on the upper-half plane $\H$.
\end{remark}

\begin{definition}[Schwarz's map]
\label{def:Schwarz's map}
For a basis $\tr(\f_1(z),\f_2(z))$ of 
$\mathcal{F}_{\dot z}(a,b,c)$,  
we have a map from a neighborhood of $\dot z=\frac{1}{2}$ 
to the complex projective line $\P^1$ by 
\begin{equation}
\label{eq:Schwarz}
\f :z\mapsto \f(z)=\frac{\f_1(z)}{\f_2(z)}
\end{equation}
Schwarz's map is define by its analytic continuation to $Z=\C-\{0,1\}$, 
which is also denoted by $\f$. 
Though Schwarz's map $\f$ is regarded as multi-valued on $Z$,  
its restriction to a simply connected domain $\C-(-\infty,0]-[1,\infty)$ 
is single valued. 
\end{definition}

\begin{fact}
\label{fact:Schwarz}
If the parameters $a,b,c$ are real and satisfy
$$|1-c|+|c-a-b|+|a-b|<1,$$
then there exists a basis $\f_1(z),\f_2(z)$ of $\mathcal{F}_{\dot z}(a,b,c)$
such that the image of Schwarz's map $\f$ is in the upper-half plane 
$\H=\{\tau\in \C\mid \im(\tau)>0\}$.  
For this basis, the image of the upper-half plane in $Z=\C-\{0,1\}$ 
under Schwarz's map $\f$ becomes 
a triangle with angles $|1-c|\pi$, $|c-a-b|\pi$, $|a-b|\pi$ 
with respect to the hyperbolic metric on $\H$, 
which is called Schwarz's triangle and denoted by $\f(\H)$. 
Moreover, $a,b,c$ satisfy 
$$\frac{1}{|1-c|},\frac{1}{|c-a-b|},\frac{1}{|a-b|}\in 
\{2,3,4,\dots\}\cup\{\infty\}$$
then the image of Schwarz's map is an open dense subset in $\H$, 
and its inverse is single valued. 
The projectivization of the monodromy representation of $\cE(a,b,c)$ 
is represented as a discrete subgroup of 
$\mathrm{PSL}_2(\R)=\mathrm{SL}_2(\R)/\{\pm I_2\}$ generated 
by two elements $g_0$ and $g_1$ with 
$$\mathrm{ord}(g_0)=\frac{1}{|1-c|}, \quad 
\mathrm{ord}(g_1)=\frac{1}{|c-a-b|}, \quad 
\mathrm{ord}(g_\infty)=\frac{1}{|a-b|},
$$
where $I_2$ is the unit matrix of size $2$, $g_\infty=(g_0g_1)^{-1}$, 
and $\mathrm{ord}(g_0)$ denotes the order of $g_0$.
\end{fact}

\section{Fundamental properties of $\h_{pq}(\tau)$}
We next introduce theta constants and their properties by referring to 
\cite{Mu}.
\begin{definition}
The theta constant $\h_{pq}(\tau)$ with characteristics $p,q$  
is defined by 
\begin{equation}
\label{eq:theta}
\h_{pq}(\tau)=\sum_{n=-\infty}^\infty \exp\big(\pi\i(n+\frac{p}{2})^2\tau
+2\pi\i(n+\frac{p}{2})\frac{q}{2}\big),
\end{equation}
where 
$\tau$ is a variable in the upper-half plane 
$\H$, and $p,q$ are parameters taking the value $0$ or $1$.
It converges absolutely and uniformly on any compact set in 
$\H$ for any fixed $p,q$.
\end{definition}

There are four functions $\h_{00}(\tau)$, $\h_{01}(\tau)$, 
$\h_{10}(\tau)$, $\h_{11}(\tau)$; 
the last one vanishes identically on $\H$, and the rests 
satisfy  Jacobi's identity
\begin{equation}
\label{eq:J-Id}
\h_{00}(\tau)^4=\h_{01}(\tau)^4+\h_{10}(\tau)^4
\end{equation}
for any $\tau\in \H$. 

The group $\SL_2(\Z)$ acts on $\H$ by the linear fractional transformation 
$$g\cdot \tau =\frac{g_{11}\tau+g_{12}}{g_{21}\tau+g_{22}},$$ 
where $g=(g_{ij})\in \SL_2(\Z)$ and $\tau\in \H$.
Note that $(-I_2)\cdot \tau=\tau$ for any $\tau \in \H$ and that 
the projectivization
$\PSL_2(\Z)=\SL_2(\Z)/\{\pm I_2\}$ acts effectively on $\H$. 

The group $\SL_2(\Z)$ is generated by two elements
\begin{equation}
T=\begin{pmatrix} 1 & 1 \\ 0 & 1\end{pmatrix},
\quad 
J=\begin{pmatrix} 0 & 1 \\ -1 & 0\end{pmatrix}.
\end{equation}
We give transformation formulas of $\h_{pq}(\tau)^2$ under the actions 
$T$ and $J$ on $\tau$:
$$T\cdot \tau=\tau+1,\quad J\cdot \tau=\frac{-1}{\tau}.$$ 
\begin{fact}[{\cite[ Table V in p.36]{Mu}}] 
\label{fact:SL2act}
We have
$$\h_{00}(\tau+1)^2=\h_{01}(\tau)^2,\quad 
\h_{01}(\tau+1)^2=\h_{00}(\tau)^2,\quad 
\h_{10}(\tau+1)^2=\i \h_{10}(\tau)^2,
$$
$$\h_{00}(\frac{-1}{\tau})^2=(-\i\tau)\h_{00}(\tau)^2,\quad 
\h_{01}(\frac{-1}{\tau})^2=(-\i\tau)\h_{10}(\tau)^2,\quad 
\h_{10}(\frac{-1}{\tau})^2=(-\i\tau)\h_{01}(\tau)^2.
$$
\end{fact}

The theta constant $\h_{00}(\tau)$ is quasi invariant under the 
action of the subgroup 
$$\Gamma_{12}=\{g=(g_{ij})\in \mathrm{SL}_2(\Z)\mid g_{11}g_{12}\equiv 
g_{21}g_{22}\equiv 0 \bmod 2\}
$$
of $\SL_2(\Z)$.

\begin{fact}[{\cite[Theorem 7.1, Table V in p.36]{Mu}}]
\label{fact:trans-theta}
Let $g=(g_{ij})$ be an element in the projectivization $\mathrm{P}\Gamma_{12}$
of 
$\Gamma_{12}$. 
By multiplying $-1$ to $g$ if necessary, 
we may assume that $g$ satisfies either 
$g_{21}>0$, or 
$g_{21}=0$ and $g_{22}>0$. 
Then we have 
$$\h_{00}(g\cdot \tau)^2
=\chi(g)\cdot (g_{21}\tau+g_{22})\h_{00}(\tau)^2
$$
for any $\tau\in \D_{12}$, where 
$$\chi(g)=\left\{
\begin{array}{lcc} \i^{g_{22}-1} &\textrm{if} & g_{21}\in 2\Z,
g_{22}\notin 2\Z,\\
\i^{-g_{21}} &\textrm{if} & g_{21}\notin 2\Z,g_{22}\in 2\Z.
\end{array}
\right.
$$
\end{fact}
Let $\Gamma(2)$ be the principal congruence subgroup 
of $\SL_2(\Z)$ of level $2$, i.e., 
$$\Gamma(2)=\{g=(g_{ij})\in \SL_2(\Z)\mid 
g_{11}-1,g_{22}-1,g_{12},g_{21}\equiv 0\bmod 2\}.
$$ 
The group $\Gamma(2)$ 
is generated by 
$T^2=\begin{pmatrix} 1 & 2 \\ 0 & 1\end{pmatrix}$, 
$J^{-1}T^2J= \begin{pmatrix} 1 & 0 \\ -2 & 1\end{pmatrix}$
and $-I_2$. 
Facts \ref{fact:SL2act} and \ref{fact:trans-theta} imply the following. 
\begin{fact}
\label{fact:Gamma2-inv}
The theta constants 
$\h_{pq}(\tau)$ $((p,q)=(0,0),(0,1),(1,0))$ satisfy 
$$\h_{pq}(g\cdot \tau)^4=(g_{21}\tau+g_{22})^2\h_{pq}(\tau)^4$$ 
for any $g=(g_{ij})\in \Gamma(2)$.  
%

\end{fact}

The group $\Gamma(2)$ is normal in $\SL_2(\Z)$, and the quotient 
$$\SL_2(\Z)/\Gamma(2)=\{I_2,W,W^2,J,T,J^{-1}TJ\},\quad 
W=(JT)^{-1}=\begin{pmatrix} -1 & -1\\ 1 & 0 \end{pmatrix}$$ 
is isomorphic to the symmetric group $S_3$ of degree $3$. 
There exists a subgroup $G$ of $\SL_2(\Z)$ such that 
the quotient $G/\Gamma(2)$ is isomorphic to 
the alternating group $A_3$ of degree $3$. 
This group is generated by $\Gamma(2)$ and $W$,
and characterized by 
$$\{g\in \SL_2(\Z)\mid g^3\in \Gamma(2)\},$$
which is denoted by $\Gamma(2)^{1/3}$. Note that 
$$[\SL_2(\Z):\Gamma(2)^{1/3}]=2,\quad [\Gamma(2)^{1/3}:\Gamma(2)]=3.$$
We give a fundamental region of $\PGa(2)^{1/3}=\Gamma(2)^{1/3}/\{\pm I_2\}$ 
and that of $\PSL_2(\Z)$ as
\begin{align}
\label{eq:fund-reg-G(1/3)}
\wt{\D}&
=\{\tau\in \H\mid -\dfrac{3}{2}< \re(\tau)\le\dfrac{1}{2},\ 
|\tau|\ge 1, \ |\tau+1|> 1\},
\\
\label{eq:fund-reg-G}
\D&
=\{\tau\in \H\mid -\dfrac{1}{2}< \re(\tau)\le \dfrac{1}{2},\ |\tau|> 1\}
\cup \{\tau\in \H\mid |\tau|=1, 0\le \re(\tau)\le \dfrac{1}{2}\},
\end{align}
see Figure \ref{fig:FD-SL2Z}.
\begin{figure}[htb]
\includegraphics[width=10cm]{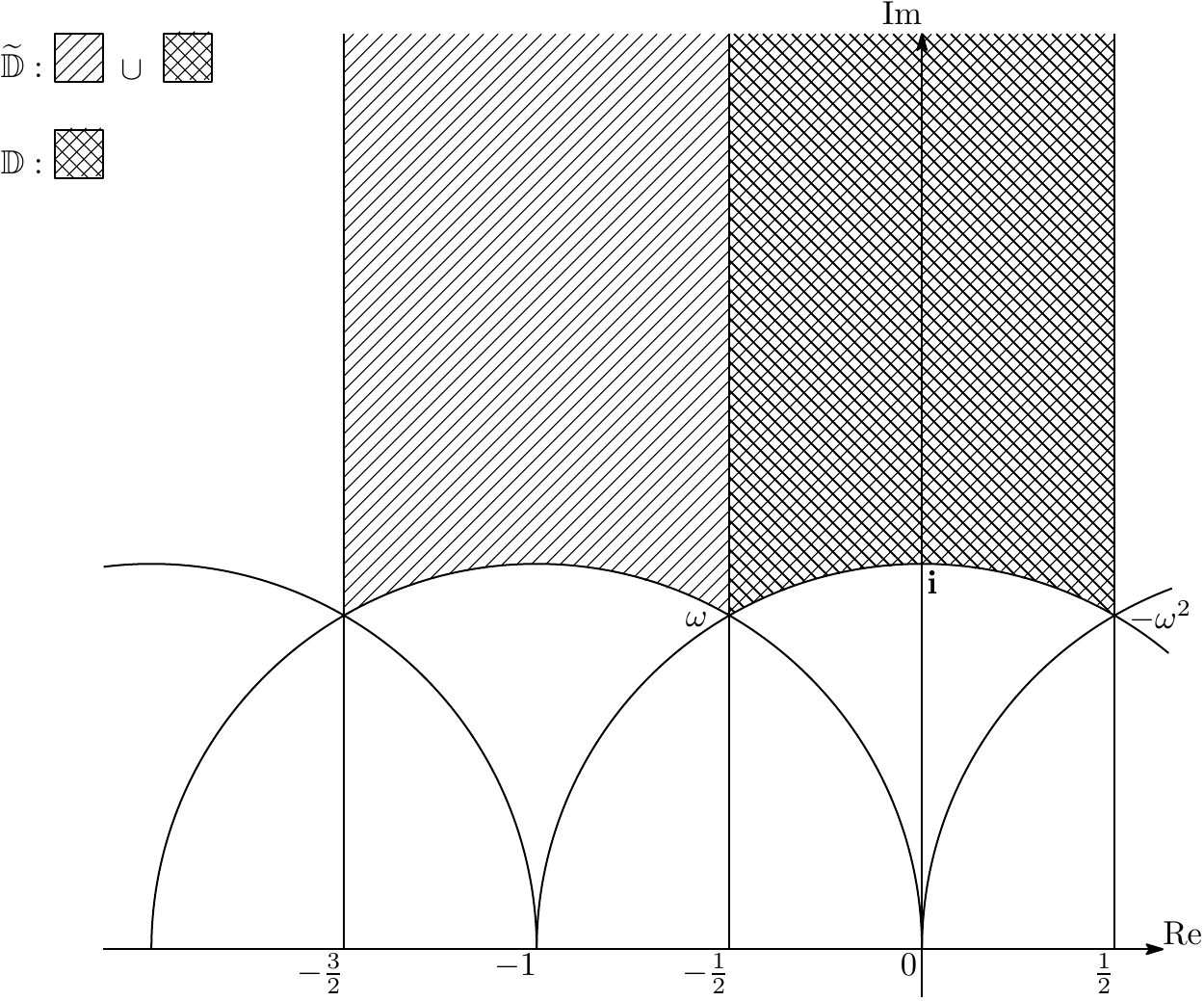}  
\caption{Fundamental regions $\wt\D$ and $\D$ 
for $\PGa(2)^{1/3}$ and $\PSL_2(\Z)$} 
\label{fig:FD-SL2Z}
\end{figure}


\begin{lemma} 
\label{lem:W-action}
By the actions of $W=\begin{pmatrix} -1 & -1\\ 1 & 0 \end{pmatrix}$ and 
$W^2=\begin{pmatrix} 0 & 1\\ -1 & -1 \end{pmatrix}
$ on $\tau\in \H$, $\h_{00}(\tau)^2$, $\h_{01}(\tau)^2$ and $\h_{10}(\tau)^2$ 
are transformed into 
\begin{equation}
\label{eq:W-action}
\begin{array}{ll}
\h_{00}(W\cdot\tau)^2=-\i \tau \cdot\h_{10}(\tau)^2, &
\h_{00}(W^2\cdot\tau)^2=-\i(\tau+1)\cdot\h_{01}(\tau)^2, 
\\
\h_{01}(W\cdot\tau)^2=-\i \tau \cdot\h_{00}(\tau)^2, &
\h_{01}(W^2\cdot\tau)^2=(\tau+1) \cdot\h_{10}(\tau)^2, 
\\
\h_{10}(W\cdot\tau)^2=-\tau \cdot \h_{01}(\tau)^2, &
\h_{10}(W^2\cdot\tau)^2=-\i(\tau+1)\cdot\h_{00}(\tau)^2.
\end{array}
\end{equation}
\end{lemma}
\begin{proof}
By Fact \ref{fact:SL2act} together with $W=T^{-1}J^{-1}$ and $J^{-1}=-J$, 
we have 
\begin{align*}
\h_{00}(W\cdot \tau)^2&=
\h_{00}(T^{-1}\cdot (J\cdot\tau))^2=\h_{01}(J\cdot\tau)^2
=(-\i \tau)\cdot \h_{10}(\tau)^2,\\
\h_{01}(W\cdot \tau)^2&=
\h_{01}(T^{-1}\cdot(J\cdot\tau))^2=\h_{00}(J\cdot\tau)^2
=(-\i \tau)\cdot \h_{00}(\tau)^2,\\
\h_{10}(W\cdot \tau)^2&=
\h_{10}(T^{-1}\cdot(J\cdot\tau))^2
=-\i\h_{10}(J\tau)^2=-\i\cdot (-\i\tau)\cdot  \h_{01}(\tau)^2. 
\end{align*}
Similarly, we can show the 
transformations 
of $\h_{pq}(\tau)^2$ under the action of $W^2=JT$. 
\end{proof}
\begin{definition}
We define the (normalized) Eisenstein series $E_{2k}(\tau)$ of weight $2k$ 
$(k\in \N)$ by 
$$E_{2k}(\tau)=\frac{1}{2\zeta(2k)}
\sum_{n_1,n_2\in\Z^2-\{(0,0)\}}\frac{1}{(n_1\tau+n_2)^{2k}},$$
where $\tau\in \H$ and $\zeta$ denotes Riemann's zeta function. 
\end{definition}

If $k\ge 2$ then this series converges absolutely, 
and it satisfies 
$$E_{2k}(g\cdot \tau)=(g_{21}\tau+g_{22})^{2k}E_{2k}(\tau)$$
for any $g\in \SL_2(\Z)$. 

\begin{fact}[{\cite[\S15]{Mu}}]
The Eisenstein series $E_4(\tau)$ is expressed by a Fourier expansion 
and by the theta constants as  
\begin{align}
\nonumber
E_4(\tau) &=1+240\sum_{n=1}^\infty \big(\sum_{d\mid n} d^3\big) e^{2\pi\i n\tau}\\
\label{eq:Fourier}
 &=1+240e^{2\pi\i \tau}+2160(e^{2\pi\i \tau})^2+6720(e^{2\pi\i \tau})^3+\cdots\\
\label{eq:fullmodular}
&=
\frac{\h_{00}(\tau)^8+\h_{01}(\tau)^8+\h_{10}(\tau)^8}{2}. 
\end{align}
\end{fact}

\section{Schwarz's map for $(a,b,c)=(\frac{1}{2},\frac{1}{2},1)$}
To study Schwarz's map $\f_0$ for parameters 
$(a,b,c)=(\frac{1}{2},\frac{1}{2},1)$, 
we select $\mathbf{F}(z)=\tr(f_1,f_2)$ as a basis of 
$\cF_{\dot z}(\frac{1}{2},\frac{1}{2},1)$.  
Schwarz's map $\f_0$ is the analytic continuation of 
\begin{equation}
\label{eq:S-Map-1/2}
\f_0: z\mapsto \tau=\f_0(z)=\frac{f_1(z)}{f_2(z)}=\i\cdot 
\frac{F(\frac{1}{2},\frac{1}{2},1,1-z)}{F(\frac{1}{2},\frac{1}{2},1,z)}
\end{equation} 
defined on $\B_0\cap \B_1=\{z\in \C\mid |z|<1,\ |1-z|<1\}$ to $Z=\C-\{0,1\}$. 
We have 
$\big(\frac{1}{|1-c|},\frac{1}{|c-a-b|},\frac{1}{|a-b|}\big)=
(\infty,\infty,\infty)$, and the circuit matrices $M_0$, $M_1$ and 
$(M_0M_1)^{-1}$ are 
$$
M_0=\begin{pmatrix}
1 & 2 \\
0 & 1
\end{pmatrix},\quad 
M_1
=\begin{pmatrix}
1 & 0 \\
-2 & 1
\end{pmatrix},\quad 
(M_0M_1)^{-1}
=\begin{pmatrix}
 1 & -2 \\
 2 & -3
 \end{pmatrix}.
$$
The group generated by these matrices is not 
the principal congruence subgroup $\Gamma(2)$
but
$$\Gamma(2,4)=\{g=(g_{ij})\in \mathrm{SL}_2(\Z)\mid 
g_{11},g_{22}\equiv 1\bmod 4, g_{12},g_{21}\equiv 0\bmod 2\}.$$
Note that $\Gamma(2,4)$ is a subgroup in $\Gamma(2)$ of index $2$, 
and that 
$\mathrm{P}\Gamma(2)=\Gamma(2)/\{\pm I_2\}$ is isomorphic to $\Gamma(2,4)$.

Since $F(\frac{1}{2},\frac{1}{2},1,1-z)$ and 
$F(\frac{1}{2},\frac{1}{2},1,z)$ take positive real values 
for $z$ in the open interval $(0,1)$, 
$\f_0(z)$ for $z\in (0,1)$ is a pure imaginary number in $\H$.
Since the image of the monodromy representation is in $\mathrm{PSL}_2(\R)$ 
and $f_1(z)/f_2(z)$ is in $\H$ for $z\in (0,1)$, 
the image of the analytic continuation of $f_1(z)/f_2(z)$ to $Z=\C-\{0,1\}$ 
is an open dense subset in $\H$.

\begin{fact}[Schwarz's map for $(a,b,c)=(\frac{1}{2},\frac{1}{2},1)$]
\label{fact:Inv-S-map}
\cite[\S 5.6,5.7,5.8 in Chap. II, Proposition 8.1 in Chap. III]{Yo}
\begin{enumerate}
\item Schwarz's triangle $\f_0(\H)$ for $(a,b,c)=(\frac{1}{2},\frac{1}{2},1)$
is 
$$\{\tau\in \H\mid 0 <\re(\tau) <1,\ |\tau-\frac{1}{2}|>\frac{1}{2}\}$$
and its vertices $\f_0(z_0)=\lim\limits_{z\to z_0,z\in \H} \f_0(z)$
$(z_0=0,1,\infty)$ are 
$$
\f_0(0)=\i\infty, \quad \f_0(1)=0, \quad \f_0(\infty)=1.
$$
The images of the intervals $(0,1)$, $(-\infty,0)$, $(1,\infty)$ 
consisting of the boundary of $\H(\subset Z=\C-\{0,1\})$ under $\f_0$ 
are $\{\tau\in \H\mid \re(\tau)=0\}$, $\{\tau\in \H\mid \re(\tau)=1\}$, 
$\{\tau\in \H\mid |\tau-\frac{1}{2}|=\frac{1}{2}\}$, respectively.
\begin{figure}[hbt]
\begin{center}
\includegraphics[width=8cm]{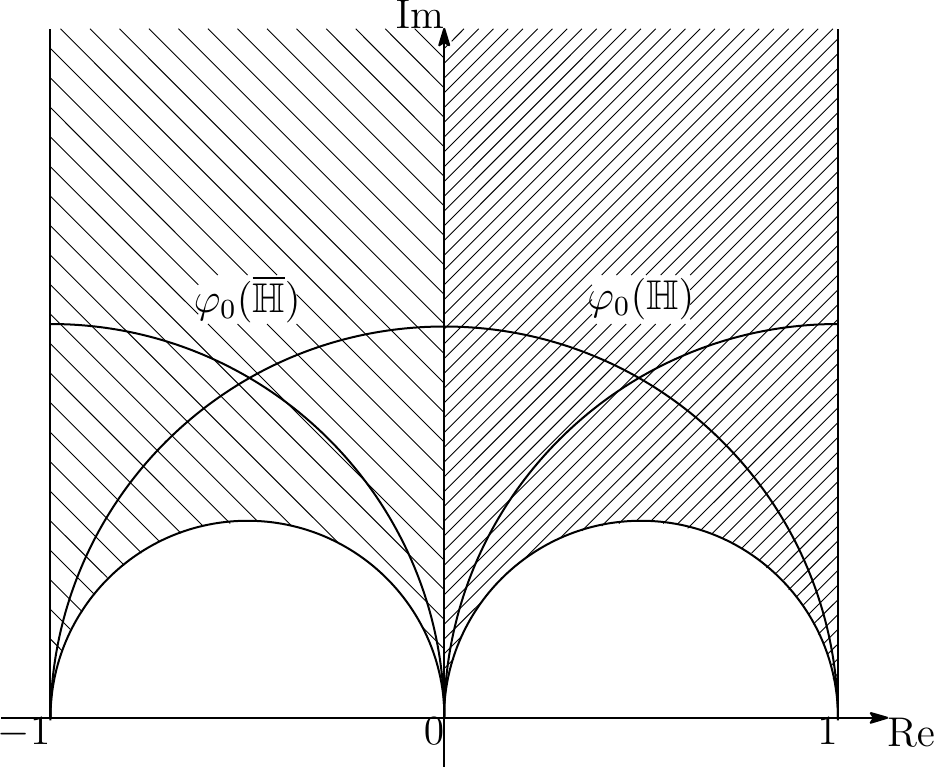}
\end{center}
\caption{Schwarz's triangle for $(a,b,c)=(\frac{1}{2},\frac{1}{2},1)$}
\end{figure}
\item 
The image $\f_0(\overline{\H})$ of the lower-half plane $\overline{\H}$ 
under $\f_0$ is the reflection of Schwarz's triangle $\f_0(\H)$ 
with respect to the imaginary axis. 
The union $\f_0(\H)^c \cup \f_0(\overline{\H})$ forms a fundamental 
region $\D(2)$ of $\PGa(2)$, where $\f_0(\H)^c$ denotes the closure 
of $\f_0(\H)$ in $\H$.

\item 
The inverse of Schwarz's map for $(a,b,c)=(\frac{1}{2},\frac{1}{2},1)$ is 
\begin{equation}
\label{eq:lambda}
\H\ni \tau \mapsto \lambda(\tau)=\frac{\h_{10}(\tau)^4}
{\h_{00}(\tau)^4}\in Z.
\end{equation} 
The function $\lambda(\tau)$ is invariant under the action 
$\tau\mapsto (g_{11}\tau+g_{12})/(g_{21}\tau+g_{22})$ 
of $g=(g_{ij})\in \PGa(2)$ on $\tau \in \H$.

\item Schwarz's map $\f_0$ and its inverse $\l$ are extended to 
isomorphisms between the complex projective line $\P^1(\supset Z=\C-\{0,1\})$ 
and the compactification $\wh{\H}/\PGa(2)$ of $\H/\PGa(2)$, where 
$\wh{\H}=\H\cup\{\i\infty\}\cup\Q$ and $\PGa(2)$ acts naturally on 
$\{\i\infty\}\cup\Q$ and there are three orbits represented by 
the vertices 
$\i\infty=\f_0(0)$, $0=\f_0(1)$, $1=\f_0(\infty)$ of Schwarz's triangle. 
These extended isomorphisms are also denoted by the same symbols $\f_0$ and 
$\l$.  

\end{enumerate}

\end{fact}

\begin{lemma}
\label{lem:midpt}
\begin{enumerate} 
\item 
The function $\l$ is transformed into 
$\l(T\cdot \tau)=\dfrac{\l(\tau)}{\l(\tau)-1}$
under the action of $T\in \SL_2(\Z)$.
\item 
The function $\l$ sends 
the mid points $\i$, $\dfrac{1+\i}{2}$, $1+\i$  
of the edges of Schwarz's triangle $\f_0(\H)$ 
to 
$$\l(\i)=\frac{1}{2},\quad \l(\frac{1+\i}{2})=2,\quad \l(1+\i)=-1.$$
\item 
The function $\l$ sends
the centers $\-\w^2$  and $\w$ of gravity  
of Schwarz's triangles $\f_0(\H)$ and $\f_0(\overline{\H})$ 
to 
$$\l(-\w^2)=-\w^2,\quad \l(\w)=-\w.$$
\end{enumerate}
\end{lemma}
\begin{proof}
(1) By Fact \ref{fact:SL2act} and \eqref{eq:J-Id}, we have
$$\l(T\cdot \tau)=\frac{\h_{10}(T\cdot \tau)^4}{\h_{00}(T\cdot \tau)^4}
=\frac{-\h_{10}(\tau)^4}{\h_{01}(\tau)^4}
=\frac{-\h_{10}(\tau)^4}{\h_{00}(\tau)^4-\h_{10}(\tau)^4}
=\frac{\l(\tau)}{\l(\tau)-1}.
$$

\smallskip\noindent
(2) Note that $\i\in \H$ is a fix point under the action of $J\in SL_2(\Z)$.
By Fact \ref{fact:SL2act} and \eqref{eq:J-Id}, 
the function $\l(\tau)$ is transformed into 
$$\l(J\cdot \tau)=\frac{\h_{10}(J\cdot \tau)^4}
{\h_{00}(J\cdot \tau)^4}=\frac{\h_{01}(\tau)^4}{\h_{00}(\tau)^4}=
\frac{\h_{00}(\tau)^4-\h_{10}(\tau)^4}{\h_{00}(\tau)^4}=1-\l(\tau)
$$
under the action of $J$. Thus we have 
$$\l(\i)=\l(J\cdot \i)=1-\l(\i),$$
which yields $\l(\i)=\dfrac{1}{2}$.
The points $\dfrac{1+\i}{2}$ and $1+\i$ in $\H$ are fixed by 
$J_1=\begin{pmatrix} -1 & 1 \\ -2 & 1\end{pmatrix}$ and 
$J_2=\begin{pmatrix} -1 & 2 \\ -1 & 1\end{pmatrix}$, which transform 
$\l(\tau)$ into 
$\dfrac{\l(\tau)}{\l(\tau)-1}$ and $\dfrac{1}{\l(\tau)}$, respectively
by Fact \ref{fact:SL2act} and \eqref{eq:J-Id} 
together with $J_1=J^{-1}T^{-1}JTJ$, $J_2=TJT^{-1}$. 
Thus we have $\l(\dfrac{1+\i}{2})=2$ and $\l(1+\i)=-1$ since 
$\l(\tau)\ne 0,1$ for any $\tau\in \H$.

\smallskip\noindent
(3)
The point $\w\in \H$ is fixed by 
$W=(JT)^{-1}$, which transform $\l(\tau)$ into 
$1-\dfrac{1}{\l(\tau)}$ 
by Lemma \ref{lem:W-action}. 
Thus we have $\l(\w)=1-\dfrac{1}{\l(\w)}$, which yields that 
$\l(\w)$ is $-\w$ or $-\w^2$.
Since $\w$ is in the left-half plane of $\C$, 
$\l(\w)$ should be $-\w$ in the lower-half plane of $\C$  
by Fact \ref{fact:Inv-S-map} (2). 
Note that $\l(\w)=-\w$ is equivalent to $\f_0(-\w)=\w$. 
Since $\overline{-\w}=-\w^2$ and the reflection image of $\w$ 
with respect to the imaginary axis is $-\w^2$, 
Fact \ref{fact:Inv-S-map} (2) also implies $\f_0(-\w^2)=-\w^2$, 
which is equivalent to $\l(-\w^2)=-\w^2$.
\end{proof}

As given in \cite[Theorem 2.1]{BB1}, 
the hypergeometric series $F(\frac{1}{2},\frac{1}{2},1;z)$ and 
the theta constants are related by Jacobi's formula:
\begin{fact}[Jacobi's formula]
\label{fact:Jacobi}
Jacobi's formula is 
\begin{equation}
\label{eq:Jacobi}
\h_{00}(\tau)^2 = F(\frac{1}{2},\frac{1}{2},1;\lambda(\tau))
\end{equation}
for any $\tau$ in the interior $\D(2)^\circ$ of 
the fundamental region $\D(2)$ of $\PGa(2)$, 
where $F(\frac{1}{2},\frac{1}{2},1;z)$ denotes  
the single-valued function on $\C-[1,\infty)$ given 
in Remark \ref{rem:discontinuous}. By acting $T$ on $\tau$ in 
\eqref{eq:Jacobi}, we have 
\begin{equation}
\label{eq:JacobiT}
\h_{01}(\tau)^2=
F(\frac{1}{2},\frac{1}{2},1;\frac{\lambda(\tau)}{\lambda(\tau)-1})
\end{equation}
for any $\tau\in \D(2)^\circ$. 
\end{fact}


\begin{remark}
\label{rem:conti-H}
Since $\H$ is simply connected,  
we can extend \eqref{eq:Jacobi} and \eqref{eq:JacobiT}
to equalities holding for any $\tau$ in the whole space $\H$. 
For these extensions, we should carefully trace the continuations of 
the right hand sides of \eqref{eq:Jacobi} and \eqref{eq:JacobiT}, 
for details refer to \cite[Corollary 4.7]{CM}.
By Schwarz's map 
$$\f_0:Z\ni z\mapsto \tau=
\i\frac{F(\frac{1}{2},\frac{1}{2},1;1-z)}{F(\frac{1}{2},\frac{1}{2},1;z)}
\in \H,$$ 
the equalities \eqref{eq:Jacobi} and \eqref{eq:JacobiT} are pulled back to  
\begin{equation}
\label{eq:Jacobi-z}
F(\frac{1}{2},\frac{1}{2},1;z)=
\h_{00}(\f_0(z))^2, \quad 
F(\frac{1}{2},\frac{1}{2},1;\frac{z}{z-1})=
\h_{01}(\f_0(z))^2,
\end{equation}
which are initially  equalities on 
$\B_0\cap\B_1$ 
and extended to those on 
the whole space $Z$ by the analytic continuation. 
\end{remark}

\section{Schwarz's map for $(a,b,c)=(\frac{1}{6},\frac{1}{2},1)$
}

We consider the hypergeometric differential equation 
$\cE(\frac{1}{6},\frac{1}{2},1)$. 
Its Riemann's scheme becomes as in Table \ref{Tab:RS621}.
\begin{table}[htb]
$$\begin{array}{|l|ccc|}
\hline
z_0 & 0 & 1 & \infty \\
\hline
e_{z_0,1}& 0 & 0 & {1}/{6}\\
e_{z_0,2}& 0 & {1}/{3} & {1}/{2}\\
\hline
e_{z_0,2}-e_{z_0,1} & {1}/{\infty} & {1}/{3} & {1}/{3}\\
\hline
  \end{array}
$$
\caption{Riemann's scheme for $(a,b,c)=(\frac{1}{6},\frac{1}{2},1)$}
\label{Tab:RS621}
\end{table}
The basis of the space of its local solutions around $\dot z=\frac{1}{2}$ 
in Lemma \ref{lem:basis} becomes 
$$\begin{pmatrix}
\ds{\exp(\frac{5\pi\i}{6})\int_{-\infty}^0(-t)^{-1/2}(z-t)^{-1/2}(1-t)^{-1/6}dt}\\
\ds{\int^{\infty}_1(t)^{-1/2}(t-z)^{-1/2}(t-1)^{-1/6}dt}\\
\end{pmatrix}=
\begin{pmatrix}
\ds{\exp(\frac{5\pi\i}{6})
B(\frac{1}{6},\frac{1}{2})F(\frac{1}{6},\frac{1}{2},\frac{2}{3},1-z)}\\[2mm]
\ds{B(\frac{1}{6},\frac{5}{6})F(\frac{1}{6},\frac{1}{2},1,z)}
\end{pmatrix}.$$
The circuit matrices with respect to this basis are 
$$M_0=\begin{pmatrix} 1 & -\w^2 \\ 0 & 1\end{pmatrix},\quad 
M_1=\begin{pmatrix} 1 & 0 \\ -2 & \w\end{pmatrix}.
$$

Since these matrices are not in $\PSL_2(\R)$, 
the analytic continuation of $f_1(z)/f_2(z)$ does not stay in $\H$. 
We find a matrix $R\in \GL_2(\C)$ such that 
$$RM_0R^{-1}, RM_1R^{-1}\in \PSL_2(\Z).$$
We can diagonalize $M_1$ by $P_x=\begin{pmatrix}
\frac{-1+\w}{2} & 0 \\ 1 & x
\end{pmatrix}$ with a parameter $x$, i.e., 
$$P_x^{-1} M_1P_x=\begin{pmatrix}
1 & 0 \\  0 & \w
\end{pmatrix}$$
holds for any $x\in \C$. 
On the other hand, since 
$$\begin{pmatrix} \w & \w^2\\ 1 &  1\end{pmatrix}^{-1} 
\begin{pmatrix} -1 & -1\\ 1 & 0 \end{pmatrix} 
\begin{pmatrix} \w & \w^2\\ 1 &  1\end{pmatrix}
=\begin{pmatrix}
\w & 0 \\  0 & \w^2
\end{pmatrix},$$
 we have 
$$
R_x M_1 R_x^{-1}=\w^2\begin{pmatrix} -1 & -1\\ 1 &  0\end{pmatrix}=N_1
$$
by setting 
$$R_x=\begin{pmatrix} \w & \w^2\\ 1 &  1\end{pmatrix}P_x^{-1}.$$
The matrix $R_x M_0R_x^{-1}$ becomes 
$$\frac{1}{3x}
\left(\begin {array}{cc}  \left( -1+\sqrt{3}\i  \right) {x}^{2}+5 x-1-\sqrt{3}\i & \left( 1+\sqrt{3}\i  \right) {x}^{2}+4 x+1-\sqrt{3}\i 
\\\noalign{\medskip}2 {x}^{2} -4 x+2& \left( 1-\sqrt{3}\i  \right) {x
}^{2}+x+1+\sqrt{3}\i \end {array}\right).
$$
Since this matrix is not diagonalizable and has an eigenvalue $1$, 
if the $(2,1)$-entry of this matrix vanishes then the action 
of this matrix becomes a translation on $\C$.  
By solving $2x^2-4x+2=0$ equivalent to the vanishing of $(2,1)$-entry of 
$R_x M_0R_x^{-1}$, we have $x=1$ and the above matrix becomes 
$N_0=\begin{pmatrix} 1 & 2 \\ 0 &  1\end{pmatrix}$ and 
$R_x$ for $x=1$ is 
$$R=R_1=\begin{pmatrix} \w & \w^2\\ 1 &  1\end{pmatrix}P_1^{-1}
=\begin{pmatrix} -2\w & \w^2 \\ 0 &  1\end{pmatrix}.
$$
Thus the circuit matrices $M_0$ and $M_1$ are transformed into 
$$RM_0R^{-1}=N_0=\begin{pmatrix} 1 & 2 \\ 0 &1\end{pmatrix},\quad 
RM_1R^{-1}=N_1=\w^2\begin{pmatrix} -1 & -1 \\ 1 &0 \end{pmatrix}.$$
The circuit matrix $N_\infty$ around $z=\infty$ is 
$$N_\infty=(N_0N_1)^{-1}=(-\w)\begin{pmatrix}
0 & -1 \\ 1 &-1
\end{pmatrix},
$$
which satisfy $N_\infty^3=-I_2$. 
Thus $N_\infty$ is of order $6$, and of order $3$ as a projective element.
Since 
$$N_1^{-1}N_0N_1=\begin{pmatrix} 1 & 0\\ -2 & 1
\end{pmatrix},$$
the group generated by $N_0$ and $N_1$ are 
$$\Gamma(2)\cup (\w^2 W\cdot \Gamma(2))\cup (\w W^2\cdot \Gamma(2)),$$
where $W=\begin{pmatrix}
-1 & -1 \\
1 & 0
\end{pmatrix}$, $W^2=\begin{pmatrix}
 0 & 1 \\
-1 & -1
\end{pmatrix}$, $W^3=I_2$ and 
\begin{align*}
\w^2 W\cdot \Gamma(2)&=\{\w^2 W g\in \GL_2(\Z[\w])\mid g\in \Gamma(2)\}, \\
\w W^2\cdot \Gamma(2)&=  \{\w W^2 g\in \GL_2(\Z[\w])\mid g\in \Gamma(2)\}. 
\end{align*}
The group $\Gamma(2)$ is normal in this group, and we have 
$$[\Gamma(2)\cup \w^2 W\cdot \Gamma(2)\cup \w W^2\cdot \Gamma(2)]/
\Gamma(2)\simeq \{I_2,W,W^2\},
$$
$$[\Gamma(2)\cup (\w^2 W\cdot \Gamma(2))\cup (\w W^2\cdot \Gamma(2))]/
\{\pm I_2,\pm\w I_2,\pm\w^2I_2\} \simeq \PGa(2)^{1/3}.$$

We change the basis $\mathbf{F}(z)$ into $R\cdot\mathbf{F}(z)$, 
then $R\cdot\mathbf{F}(z)$ takes the form 
$$\begin{pmatrix}
{-2\w f_1(z)+\w^2f_2(z) }\\{f_2(z)}
\end{pmatrix}
=\begin{pmatrix}
{2\i B(\frac{1}{6},\frac{1}{2})F(\frac{1}{6},\frac{1}{2},\frac{2}{3};1-z)
+2\w^2\pi F(\frac{1}{6},\frac{1}{2},1;z)}\\
{2 \pi F(\frac{1}{6},\frac{1}{2},1;z)}
\end{pmatrix}.
$$
\begin{definition}[Schwarz's map $\f_1$ for 
$(a,b,c)=(\frac{1}{6},\frac{1}{2},1)$]      
We define Schwarz's map $\f_1$ for $(a,b,c)=(\frac{1}{6},\frac{1}{2},1)$ 
by the analytic continuation of the map 
\begin{align*}
\f_1:z\mapsto \tau&
=\frac{2\i 
B(\frac{1}{6},\frac{1}{2})F(\frac{1}{6},\frac{1}{2},\frac{2}{3};1-z)
+2\w^2\pi F(\frac{1}{6},\frac{1}{2},1;z)}
{2 \pi F(\frac{1}{6},\frac{1}{2},1;z)}\\
& =\w^2+\i\cdot  \frac{B(\frac{1}{6},\frac{1}{2})}{\pi} \cdot 
\frac{F(\frac{1}{6},\frac{1}{2},\frac{2}{3};1-z)}
{F(\frac{1}{6},\frac{1}{2},1;z)}
\end{align*}
defined on $\B_0\cap \B_1$ to $\C-\{0,1\}$. 
Its projective monodromy representation is isomorphic to $\PGa(2)^{1/3}$.
\end{definition}

\begin{proposition}
\label{prop:ST33inf}
Schwarz's triangle $\f_1(\H)$ for $(a,b,c)=(\frac{1}{6},\frac{1}{2},1)$ is 
$$\{\tau\in \H\mid -\frac{1}{2} <\re(\tau) <\frac{1}{2},\ 
|\tau|>1\},$$
which is the interior of the fundamental domain $\D$ for $\PSL_2(\Z)$,  
and its vertices 
$\f_1(z_0)=\lim\limits_{z\to z_0,z\in \H} \f_1(z)$
$(z_0=0,1,\infty)$ are 
$$
\f_1(0)=\i\infty, \quad \f_1(1)=\w, \quad \f_1(\infty)=-\w^2.
$$
The images of the intervals $(0,1)$, $(-\infty,0)$, $(1,\infty)$ 
consisting of the boundary of $\H(\subset Z=\C-\{0,1\})$ under $\f_1$ 
are 
\begin{align*}
&\{\tau\in \H\mid \re(\tau)=-\frac{1}{2},\  
\im(\tau)>\frac{\sqrt{3}}{2}\},\\ 
&\{\tau\in \H\mid \re(\tau)=\frac{1}{2},\  
\im(\tau)> \frac{\sqrt{3}}{2}\},\\ 
&\{\tau\in \H\mid |\tau|=1,\  |\re(\tau)|<\frac{1}{2}\},
\end{align*}
respectively. See Figure \ref{fig:FD-SL2Z} for their shape. 
\end{proposition}
\begin{proof}
We firstly consider the image of $(0,1)$ under $\f_1$. 
Note that $F(\frac{1}{6},\frac{1}{2},\frac{2}{3};1-z)$ and 
$F(\frac{1}{6},\frac{1}{2},1;z)$ take positive real values. 
Since 
$$\lim_{z\to 0,z\in (0,1)}F(\frac{1}{6},\frac{1}{2},\frac{2}{3};1-z)=\infty 
, \quad 
\lim_{z\to 0,z\in (0,1)}F(\frac{1}{6},\frac{1}{2},1;z)=1
$$
by Fact \ref{fact:G-K} for $(a,b,c)=(\frac{1}{6},\frac{1}{2},\frac{2}{3})$
together with $c-a-b=0$,  we have $\f_1(0)=\i\infty$. 
Since 
$$\lim_{z\to 1,z\in (0,1)}F(\frac{1}{6},\frac{1}{2},\frac{2}{3};1-z)=1,  \quad 
\lim_{z\to 1,z\in (0,1)}F(\frac{1}{6},\frac{1}{2},1;z)
=\frac{\G(1)\G(1-\frac{1}{6}-\frac{1}{2})}{\G(1-\frac{1}{6})\G(1-\frac{1}{2})}
=\frac{\G(1)\G(\frac{1}{3})}{\G(\frac{5}{6})\G(\frac{1}{2})}
$$
by Fact \ref{fact:G-K} for $(a,b,c)=(\frac{1}{6},\frac{1}{2},1)$
together with $c-a-b=\frac{1}{3}>0$, 
we have
\begin{align*} 
\f_1(1)&=\w^2+\i\cdot \frac{B(\frac{1}{6},\frac{1}{2})}{\pi}
\cdot  \frac{\G(\frac{5}{6})\G(\frac{1}{2})}{\G(1)\G(\frac{1}{3})}
=\w^2+\i\cdot\frac{\G(\frac{1}{6})\G(\frac{1}{2})\G(\frac{5}{6})\G(\frac{1}{2})}
{\pi\G(\frac{2}{3})\G(\frac{1}{3})}\\
&=\w^2+\i\cdot\frac{\sin\frac{\pi}{3}}{\sin\frac{\pi}{6}}=\w^2+\sqrt{3}\i=\w,
\end{align*}
where we use the reflection formula 
$\G(a)\G(1-a)=\dfrac{\pi}{\sin (\pi a)}$ for the Gamma function. 

We secondly consider the image $\f_1(\infty)$. 
Note that the eigenvalues of $N_\infty$ are $-\w^2$ and $-1$ and that 
their eigen row vectors are $(\w^2,1)$ and $(\w,1)$, respectively. 
Thus $(\w^2,1)\cdot R\cdot\mathbf{F}(z)$ and $(\w,1)\cdot R\cdot\mathbf{F}(z)$ 
take forms of 
$$
\a(1/z)^{1/6}f_{\infty,1}(z), 
\quad 
\b(1/z)^{1/2}f_{\infty,2}(z)
$$
by Table \ref{Tab:RS621},  
where $\a$ and $\b$ are non-zero complex numbers, and 
$f_{\infty,i}(z)$ $(i=1,2)$ are holomorphic functions around $z=\infty$. 
Thus $R\cdot \mathbf{F}(z)$ 
admits an expression 
\begin{equation}
\label{eq:conn}
\begin{array}{ll}
R\cdot \mathbf{F}(z)&=
\begin{pmatrix} \w^2 & 1\\ \w &1 \end{pmatrix}^{-1}
\begin{pmatrix}
\a(1/z)^{1/6}f_{\infty,1}(z) \\ \b(1/z)^{1/2}f_{\infty,2}(z)
\end{pmatrix}\\
&=
\dfrac{1}{\w^2-\w} 
\begin{pmatrix}
\a(1/z)^{1/6}f_{\infty,1}(z)-\b(1/z)^{1/2}f_{\infty,2}(z)\\
-\w\a(1/z)^{1/6}f_{\infty,1}(z)+\w^2\b(1/z)^{1/2}f_{\infty,2}(z)
\end{pmatrix},
\end{array}
\end{equation}
which implies 
\begin{align*}\f_1(\infty)&=\lim_{z\to \infty} 
\frac{\a(1/z)^{1/6}f_{\infty,1}(z)-\b(1/z)^{1/2}f_{\infty,2}(z)}
{-\w \a(1/z)^{1/6}f_{\infty,1}(z) + \w^2 \b(1/z)^{1/2}f_{\infty,2}(z)}\\
&=\lim_{z\to \infty} 
\frac{\a(1/z)^{1/6}f_{\infty,1}(z)}
{-\w\a(1/z)^{1/6}f_{\infty,1}(z)}=-\w^2.
\end{align*}

We finally determine the images of $(-\infty,0)$ and $(1,\infty)$ under 
$\f_1$. 
Since each of them is a part of a circle or line 
perpendicular to the real axis,  
$\f_1(-\infty,0)$ is the line from $-\w^2$ to $\i\infty$ 
parallel to the imaginary axis, and 
$\f_1(1,\infty)$ is the arc from $\w$ to $-\w^2$ with radius $1$ 
and center at the origin. 
\end{proof}

\begin{remark}
The image of $Z=\C-\{0,1\}$ under $\f_1$ is 
$$\H_1=\H-\{g\cdot\w,\ g\cdot(-\w^2)\in \H\mid g\in \PGa(2)^{1/3}\}.$$
We can extend Schwarz's map $\f_1$ to a map from $\P^1-\{0\}$ to 
$\H$ by corresponding $1$ and $\infty$ to 
the vertices $\w=\f_1(1)$ and $-\w^2=\f_1(\infty)$ of Schwarz's triangles 
or their equivalent points under the action of $\PGa(2)^{1/3}$. 
Moreover, we can extend it to a map from $\P^1$ to 
$\wh{\H}=\H\cup\{\i\infty\}\cup \Q$. 
Note that any point in  $\Q$ is equivalent to $\i\infty=\f_1(0)$ 
under the action of $\PGa(2)^{1/3}$. In fact, for any irreducible fraction
$r/s$, there exists $r',s'\in \Z$ such that $rr'+ss'=1$. 
Thus $g=\begin{pmatrix} r & -s'\\ s & r'
\end{pmatrix}$ belongs to $\SL_2(\Z)$ and sends $\i\infty$ to $r/s$. 
If this matrix does not belong to $\Gamma(2)^{1/3}$ then 
$g\cdot T$ belongs to $\Gamma(2)^{1/3}$ and sends $\i\infty$ to $r/s$, 
since $T\notin \Gamma(2)^{1/3}$ fixes  $\i\infty$ and 
$[\SL_2(\Z):\Gamma(2)^{1/3}]=2$.
%
These extensions of Schwarz's map $\f_1$ are denoted by 
the same symbol $\f_1$.
\end{remark}
\begin{remark}
To obtain $\f_1(\infty)$ in Proof of Proposition \ref{prop:ST33inf}, 
we cannot use 
$N_\infty'=(N_1N_0)^{-1}=N_0^{-1}N_1^{-1}$ which is a circuit matrix 
with respect to $R\cdot \mathbf{F}(z)$ along the loop connecting 
the end of $\rho_0^{-1}$ and the start of $\rho_1^{-1}$. 
This matrix is useful to study the reflection of Schwarz's triangle
with respect to the image of $(0,1)$ under $\f_1$. 
\end{remark}

\begin{theorem}
\label{th:inv-Sch}
The inverse of Schwarz's map $\f_1$ is given by 
\begin{equation}
\label{eq:nu}
\nu(\tau)=3\sqrt{3}\i \frac{\h_{00}(\tau)^4\h_{01}(\tau)^4\h_{10}(\tau)^4}
{(\h_{00}(\tau)^4+\w\h_{10}(\tau)^4)^3}
.
\end{equation}
\end{theorem}
\begin{proof} Schwarz's map $\f_1$ leads to an isomorphism between 
$Z$ and $\H_1/\PGa(2)^{1/3}$. 
Since $\PGa(2)^{1/3}/\PGa(2)\simeq \{I_2,W,W^2\}\simeq A_3$, 
the space $\H_1/\PGa(2)^{1/3}$ is regarded as the quotient of $\H/\PGa(2)$ 
by the cyclic group $\{I_2,W,W^2\}$. Here 
$\H/\PGa(2)$ is isomorphic to $Z$ by 
$$\l :\H/\PGa(2)\ni \tau \mapsto \l(\tau)
=\frac{\h_{10}(\tau)^4}{\h_{00}(\tau)^4}\in Z.
$$
To distinguish the space $\C-\{0,1\}$ isomorphic to 
$\H_1/\PGa(2)^{1/3}$ and that to $\H/\PGa(2)$, 
the former is denoted by $Z_\nu$ and the latter by $Z_\l$.
Note that $Z_\l$ is a cyclic triple covering of $Z_\nu$. 
By tracing the action of $W$ via the map $\l$, 
we characterize the covering transformation group for the projection 
$\pr_{\l\nu}:Z_\l\to Z_\nu$. 
The function $\l(\tau)$ is transformed into 
\begin{equation}
\label{eq:W-act-lambda}
\l(W\cdot \tau)=1-\dfrac{1}{\l(\tau)},\quad 
\l(W^2\cdot \tau)=1-\dfrac{1}{\l(W\cdot \tau)}=\dfrac{1}{1-\l(\tau)}
\end{equation}
by the actions of $W$ and $W^2$ on $\tau$ as 
shown in Proof of Lemma \ref{lem:midpt} (3).  
Thus the covering transformation
group is generated by a map 
$$\sigma:Z_\l\ni \l\mapsto 1-\frac{1}{\l}\in Z_\l$$ 
satisfying 
$$\sigma^2(\l)=\frac{1}{1-\l},\quad \sigma^3(\l)=\l.$$  
By an invariant function 
$$\nu^\sigma(\l)=\frac{\l+\sigma(\l)+\sigma^2(\l)}{3}=
\frac{1}{3}\left(\l+(1-\frac{1}{\l})+\frac{1}{1-\l}\right)=
\frac{\l^3-3\l+1}{3\l(\l-1)}$$
under $\sigma$, we have a projection 
$$
\pr^\sigma_{\nu} :Z_\l\ni \l \mapsto \nu^\sigma(\l)\in Z_\l/\la \sigma\ra,$$
where  $Z_\l/\la \sigma\ra$ is isomorphic to $Z_\nu$.
We find an isomorphism $\psi:Z_\l/\la \sigma\ra\to Z_\nu$ 
such that $\psi(\nu^\sigma(\l(\tau)))$ becomes the inverse of $\f_1$.
The fixed points in $Z_\l$ of $\sigma$ are $-\w^2$ and $-\w$,   
which are solutions to $\sigma(\l)=\l$ equivalent to $\l^2-\l+1=0$. 
Since these points are fixed by $\sigma^2$, these points are 
ramification points of index $3$ of $\pr_{\nu}:Z_\l\to Z_\nu$, 
and the function $\nu^\sigma$ sends these points to 
$$\nu^\sigma(-\w^2)=-\w^2,\quad \nu^\sigma (-\w)=-\w.$$
On the other hand, the preimages of these points under the map 
$\l(\tau):\H/\PGa(2)\to Z_\l$ are $-\w^2$ and $\w$ by 
Fact \ref{fact:Inv-S-map}, i.e. 
$$\l(-\w^2)=-\w^2,\quad \l(\w)=-\w.$$ 
Hence the map $\nu^\sigma(\l(\tau))$ sends $\tau=\i\infty,\w,-\w^2$ as  
$$\nu^\sigma(\l(\i\infty))=\infty,\quad \nu^\sigma(\l(\w))=-\w,
\quad \nu^\sigma(\l(-\w^2))=-\w^2.$$
Here recall that 
$$\f_1(0)=\i\infty,\quad \f_1(1)=\w,\quad \f_1(\infty)=-\w^2.$$
Since $Z_\l/\la \sigma\ra$ and $Z_\nu$ are 
the complex projective line $\P^1$ minus $3$ points, 
the inverse of $\f_1$ is give as $\psi(\nu(\l(\tau)))$ 
by the linear fractional transformation $\psi$ sending 
$\infty$, $-\w$, $-\w^2$ to $0$, $1$, $\infty$. 
In fact, $\psi$ is given by 
$$\psi(z)=\frac{\w^2-\w}{z+\w^2}$$
and 
$$\psi(\nu^\sigma(\l))
=-\sqrt{3}\i\frac{3\l(\l-1)}{\l^3-3\l+1+3\w^2\l(\l-1)}
=\frac{3\sqrt{3}\i \l(1-\l)}{(\l+\w^2)^3}.
$$
By substituting $\l=\dfrac{\h_{10}(\tau)^4}{\h_{00}(\tau)^4}$  into 
the last term and using Jacobi's identity \eqref{eq:J-Id}, 
we have the expression $\nu(\tau)$ of $\f_1^{-1}$. 
\end{proof}

\begin{theorem}
\label{th:Ana1-Jacobi}
We have 
\begin{equation}
\label{eq:J621}
\h_{00}(\tau)^4+\w\h_{10}(\tau)^4=
F(\frac{1}{6},\frac{1}{2},1;\nu(\tau))^2
\end{equation}
for any $\tau$ in the interior $\wt\D^\circ$ of $\wt\D$, 
where $\nu(\tau)$ is given in \eqref{eq:nu}.
It is extended to an equality on the whole space $\H$.
\end{theorem}

\begin{proof}
We first show that the both sides of \eqref{eq:J621} 
behave in the same manner as holomorphic functions on $\H$  
under the action of $N_0$ and $N_1$.
It is easy to see that the left hand side $\h_{00}(\tau)^4+\w\h_{10}(\tau)^4$ of \eqref{eq:J621} is invariant under the action of $N_0$ by 
Fact \ref{fact:SL2act} together with $N_0=T^2$. 
Since $N_1=\w^2W$ acts on it projectively, we may use $W$ instead of $N_1$. 
By \eqref{eq:W-action} and Jacobi's identity \eqref{eq:J-Id}, we have 
$$\h_{00}(W\cdot \tau)^4+\w\h_{10}(W\cdot \tau)^4=
-\tau^2\h_{10}(\tau)^4+\w \tau^2\h_{01}(\tau)^4
=\w\tau^2(\h_{00}(\tau)^4+\w\h_{10}(\tau)^4).
$$
We study the right hand side $F(\frac{1}{6},\frac{1}{2},1;\nu(\tau))^2$ 
of \eqref{eq:J621} under the actions of $N_0$ and $N_1$. 
Though the function $\nu(\tau)$ is invariant under the actions of 
$N_0$ and $N_1$, we should study the analytic continuation of 
$F(\frac{1}{6},\frac{1}{2},1;\nu(\tau))^2$ along paths $\g_0$ and $\g_1$ 
from $\tau$ to $N_0\cdot \tau=\tau+2$ and to $N_1\cdot \tau=1-\frac{1}{\tau}$ 
since $F(\frac{1}{6},\frac{1}{2},1;z)$ is not single valued.  
Note that the function $\nu$ sends these paths to loops homotopic to $\rho_0$ 
and $\rho_1$ in $Z=\C-\{0,1\}$. Thus 
$2\pi F(\frac{1}{6},\frac{1}{2},1;\nu(\tau))$ is invariant 
under the analytic continuation along $\g_0$,  
and transformed into  $\w^2$ times the first entry of 
$R\cdot \mathbf{F}(\nu(\tau))$
by the analytic continuation along $\g_1$, 
since $2\pi F(\frac{1}{6},\frac{1}{2},1;z)$ is the second entry of 
$R\cdot \mathbf{F}(\nu(\tau))$, and 
the second rows of $N_0$ and $N_1$ are $(0,1)$ and $(\w^2,0)$. 
Recall that $\tau$ is the quotient of the first entry of 
$R\cdot \mathbf{F}(\nu(\tau))$ by the second. 
Thus $F(\frac{1}{6},\frac{1}{2},1;\nu(\tau))^2$ is transformed into 
$$\big(\w^2\tau\cdot 
F(\frac{1}{6},\frac{1}{2},1;\nu(\tau))\big)^2
=\w\tau^2\cdot 
F(\frac{1}{6},\frac{1}{2},1;\nu(\tau))^2
$$
by the analytic continuation along $\g_1$.

We next study the zeros of the both sides of \eqref{eq:J621}. 
The function 
$\h_{00}(\tau)^4+\w\h_{10}(\tau)^4$
appears in the denominator of $\nu(\tau)$. 
Since $\nu(\tau)$ has a simple zero at $\tau=\i\infty$ and 
a simple pole at $\tau=-\w^2$ as a function on $\wh \H/\PGa(2)^{1/3}$, 
and the factor $\h_{00}(\tau)\h_{01}(\tau)\h_{10}(\tau)$ in its numerator 
never vanishes on $\H$, 
$\h_{00}(\tau)^4+\w\h_{10}(\tau)^4$ has a simple zero at 
$\tau=-\w^2$ as a function on $\H$, and never vanishes at any point 
not in the $\PGa(2)^{1/3}$-orbit of $-\w^2$.  
Since $\tau\in \H$ is the quotient of the first entry of 
$R\cdot \mathbf{F}(z)$ by the second, if its denominator 
$2\pi F(\frac{1}{6},\frac{1}{2},1;z)$ vanishes at $z_0\in Z$, 
then its numerator also vanishes at this point. By the multi-linearity 
of the determinant, the Wronskian of  $\mathbf{F}(z)$ should has 
a factor $z-z_0$. However, this contradicts \eqref{eq:Wronskian1}, 
which admits the analytic continuation along any path in $Z$.
Thus if $\tau\in \H$ does not belong to the $\PGa(2)^{1/3}$-orbit of $-\w^2$ or 
the $\PGa(2)^{1/3}$-orbit of $\w$, 
then $F(\frac{1}{6},\frac{1}{2},1;\nu(\tau))^2$ does not vanish.
Since $\nu(\w)=1$ and  
$$F(\frac{1}{6},\frac{1}{2},1;1)=
\frac{\G(\frac{1}{3})}{\G(\frac{5}{6})\G(\frac{1}{2})}\ne 0$$ 
by \eqref{eq:G-K}, 
$F(\frac{1}{6},\frac{1}{2},1;\nu(\w))^2$
does not vanish at $\tau=\w$. 
Since $\nu(\tau)$ has a pole at $\tau=-\w^2$ with order $3$ as 
a function on $\H$, and the behavior of 
$F(\frac{1}{6},\frac{1}{2},1;z)$ around $z=\infty$ is 
expressed by Landau's symbol $O((1/z)^{1/6})$ $(z=\infty)$ by \eqref{eq:conn},
$F(\frac{1}{6},\frac{1}{2},1;\nu(\w))^2$ has a simple zero at $\tau=-\w^2$ 
as a function on $\H$. 

We finally show \eqref{eq:J621}. As shown until now, 
the ratio 
$$
\frac{F(\frac{1}{6},\frac{1}{2},1;\nu(\tau))^2}
{\h_{00}(\tau)^4+\w\h_{10}(\tau)^4}
$$
is regarded as a holomorphic function on the quotient space 
$\H/\PGa(2)^{1/3}$, 
whose one-point compactification by $\i\infty$ is isomorphic to $\P^1$.  
Since 
$$\lim_{\tau\to \i\infty}(\h_{00}(\tau)^4+\w\h_{10}(\tau)^4)=1,
\quad 
\lim_{\tau\to \i\infty}F(\frac{1}{6},\frac{1}{2},1;\nu(\tau))^2
=F(\frac{1}{6},\frac{1}{2},1;0)^2=1,
$$
the ratio becomes a holomorphic function on $\P^1$. 
It should be a constant, and this constant is $1$.
\end{proof}

\begin{remark}
\label{rem:multi-valued-1}
Though $F(\frac{1}{6},\frac{1}{2},1;\nu(\tau))^2$ is a single-valued 
holomorphic function on $\H$, 
$F(\frac{1}{6},\frac{1}{2},1;\nu(\tau))$ is not. 
As shown in Proof of Theorem \ref{th:Ana1-Jacobi}, 
any element of the $\PGa(2)^{1/3}$-orbit 
$$\{g\cdot (-\w^2)\in \H\mid g\in \PGa(2)^{1/3}\}$$ 
of $-\w^2$ is a simple zero of 
$F(\frac{1}{6},\frac{1}{2},1;\nu(\tau))^2
=\h_{00}(\tau)^4+\w\h_{10}(\tau)^4$. 
Note also that the monodromy representation for 
$\cE(\frac{1}{6},\frac{1}{2},1)$ includes the scalar multiple of $-1$. 
\end{remark}

\begin{cor}
\label{cor:Ana1-Jacobi}
We have 
\begin{equation}
\label{eq:J621-inv}
\begin{array}{ll}
\h_{01}(\tau)^4-\w\h_{10}(\tau)^4&=
F(\dfrac{1}{6},\dfrac{1}{2},1;-3\sqrt{3}\i 
\dfrac{\h_{00}(\tau)^4\h_{01}(\tau)^4\h_{10}(\tau)^4}
{(\h_{01}(\tau)^4-\w\h_{10}(\tau)^4)^3})^2\\[2mm]
&=F(\dfrac{1}{6},\dfrac{1}{2},1;\dfrac{\nu(\tau)}{\nu(\tau)-1})^2
\end{array}
\end{equation}
for any $\tau\in \wt\D^\circ$, where $\nu(\tau)$ is given in \eqref{eq:nu}.
They are extended to equalities on the whole space $\H$.
\end{cor}
\begin{proof}
We have only to act $T$ on the equality \eqref{eq:J621}. 
By Fact \ref{fact:SL2act}, the left hand side of \eqref{eq:J621} 
is transformed into 
$$\h_{00}(T\cdot \tau)^4+\w\h_{10}(T\cdot \tau)^4
=\h_{01}(\tau)^4-\w\h_{10}(T\cdot \tau)^4. 
$$
The function $\nu(\tau)$ in the right hand side of \eqref{eq:J621} 
is transformed into 
$$
\nu(T\cdot \tau)=
3\sqrt{3}\i \frac{\h_{00}(T\cdot \tau)^4\h_{01}(T\cdot \tau)^4
\h_{10}(T\cdot \tau)^4}{(\h_{00}(T\cdot \tau)^4+\w\h_{10}(T\cdot \tau)^4)^3}
=3\sqrt{3}\i \frac{-\h_{00}(\tau)^4\h_{01}(\tau)^4\h_{10}(\tau)^4}
{(\h_{01}(\tau)^4-\w\h_{10}(\tau)^4)^3}.
$$

Since $\l(T\cdot \tau)=\dfrac{\l(\tau)}{\l(\tau)-1}$, 
\begin{align*}
\nu(T\cdot \tau)&=\frac{3\sqrt{3}\i \l(T\cdot \tau)(1-\l(T\cdot \tau))}
{(\l(T\cdot \tau)+\w^2)^3}=
\frac{3\sqrt{3}\i \frac{\l(\tau)}{\l(\tau)-1}(1-\frac{\l(\tau)}{\l(\tau)-1})}
{(\frac{\l(\tau)}{\l(\tau)-1}+\w^2)^3}
=
\frac{3\sqrt{3}\i \l(\tau)(1-\l(\tau))}
{(-\w\l(\tau)-1)^3},
\\
\frac{\nu(\tau)}{\nu(\tau)-1}&=
\frac{3\sqrt{3}\i \l(\tau)(1-\l(\tau))}
{(\l(\tau)+\w^2)^3}
\cdot \frac{(\l(\tau)+\w^2)^3}
{3\sqrt{3}\i \l(\tau)(1-\l(\tau))-(\l(\tau)+\w^2)^3}\\
&=\frac{3\sqrt{3}\i \l(\tau)(1-\l(\tau))}
{-\l(\tau)^3-3(\w^2+\sqrt{3}\i)\l(\tau)^2-3(\w-\sqrt{3}\i)\l(\tau)-1}
=\frac{3\sqrt{3}\i \l(\tau)(1-\l(\tau))}{-(\l(\tau)+\w)^3},
\end{align*}
we have 
$\nu(T\cdot \tau)=\dfrac{\nu(\tau)}{\nu(\tau)-1}.$
\end{proof}

\begin{remark}
\label{rem:rel-theta's}
By Jacobi's identity \eqref{eq:J-Id}, we have 
\begin{align*}
\h_{00}(\tau)^4+\w\h_{10}(\tau)^4&=
\h_{01}(\tau)^4-\w^2\h_{10}(\tau)^4=
-\w^2\h_{00}(\tau)^4-\w\h_{01}(\tau)^4\\
&=\frac{1-\w^2}{3}\left(\h_{00}(\tau)^4-\w\h_{01}(\tau)^4-\w^2\h_{10}(\tau)^4 
\right),\\
\h_{00}(\tau)^4+\w^2\h_{10}(\tau)^4&=
\h_{01}(\tau)^4-\w\h_{10}(\tau)^4=
-\w\h_{00}(\tau)^4-\w^2\h_{01}(\tau)^4\\
&=\frac{1-\w}{3}\left(\h_{00}(\tau)^4-\w^2\h_{01}(\tau)^4-\w\h_{10}(\tau)^4 
\right), 
\end{align*}
\begin{align*}
E_4(\tau)&=
\frac{\h_{00}(\tau)^8+\h_{10}(\tau)^8+\h_{01}(\tau)^8}{2}=
\h_{00}(\tau)^8-\h_{00}(\tau)^4\h_{10}(\tau)^4+\h_{10}(\tau)^8\\
&=\h_{00}(\tau)^8-\h_{00}(\tau)^4\h_{01}(\tau)^4+\h_{01}(\tau)^8
=\h_{01}(\tau)^8+\h_{01}(\tau)^4\h_{10}(\tau)^4+\h_{10}(\tau)^8.
\end{align*}
\end{remark}

\begin{cor}
\label{cor:8jyouwa}
For any $\tau\in \wt\D^\circ$, we have 
\begin{equation}
\label{eq:JF-prod-type}
E_4(\tau)=\frac{\h_{00}(\tau)^8+\h_{01}(\tau)^8+\h_{10}(\tau)^8}{2}=
F(\frac{1}{6},\frac{1}{2},1;\nu(\tau))^2
F(\frac{1}{6},\frac{1}{2},1;\frac{\nu(\tau)}{\nu(\tau)-1})^2,
\end{equation}
which is extended to an equality on the whole space $\H$.
\end{cor}
\begin{proof}
We have 
\begin{align*}
&F(\frac{1}{6},\frac{1}{2},1;\nu(\tau))^2
F(\frac{1}{6},\frac{1}{2},1;\frac{\nu(\tau)}{\nu(\tau)-1})^2
=
(\h_{00}(\tau)^4+\w\h_{10}(\tau)^4)(\h_{00}(\tau)^4+\w^2\h_{10}(\tau)^4)\\
=&\h_{00}(\tau)^8-\h_{00}(\tau)^4\h_{10}(\tau)^4+\h_{10}(\tau)^8
=\frac{\h_{00}(\tau)^8+\h_{10}(\tau)^8+\h_{01}(\tau)^8}{2}=E_4(\tau)
\end{align*}
by Theorem \ref{th:Ana1-Jacobi}, 
Corollary \ref{cor:Ana1-Jacobi}  and  Remark \ref{rem:rel-theta's}. 
\end{proof}

By pulling back identities in 
Theorem \ref{th:Ana1-Jacobi} and Corollaries 
\ref{cor:Ana1-Jacobi} and \ref{cor:8jyouwa} under Schwarz's  map 
$\f_1:Z\to \H_1$, we have the following.
\begin{cor}
\label{cor:Ana1-Jacobi-B}
We have 
\begin{align}
\nonumber
F(\frac{1}{6},\frac{1}{2},1;z)^2&=
 \h_{00}(\f_1(z))^4+\w\h_{10}(\f_1(z))^4,\\
%
\label{eq1:Jocobi-1-z}
F(\frac{1}{6},\frac{1}{2},1;\frac{z}{z-1})^2&=
 \h_{00}(\f_1(z))^4+\w^2\h_{10}(\f_1(z))^4,\\
%
\nonumber
F(\frac{1}{6},\frac{1}{2},1;z)^2F(\frac{1}{6},\frac{1}{2},1;\frac{z}{z-1})^2
&=\frac{\h_{00}(\f_1(z))^8+\h_{01}(\f_1(z))^8+\h_{10}(\f_1(z))^8}{2}=E_4(\f_1(z)),
\end{align}
which are initially equalities on 
$\B_0\cap\B_1$ 
and extended to those on 
the whole space $Z$ by the analytic continuation. 
\end{cor}

\section{Schwarz's map for $(a,b,c)=(\frac{1}{12},\frac{5}{12},1)$
}

We consider the hypergeometric differential equation 
$\cE(\frac{1}{12},\frac{5}{12},1)$. 
Its Riemann's scheme becomes as in Table \ref{Tab:RS1512}.
\begin{table}[htb]
$$\begin{array}{|l|ccc|}
\hline
z_0 & 0 & 1 & \infty \\
\hline
e_{z_0,1}& 0 & 0 & {1}/{12}\\
e_{z_0,2}& 0 & {1}/{2} & {5}/{12}\\
\hline
e_{z_0,2}-e_{z_0,1} & {1}/{\infty} & {1}/{2} & {1}/{3}\\
\hline
  \end{array}
$$
\caption{Riemann's scheme for $(a,b,c)=(\frac{1}{12},\frac{5}{12},1)$}
\label{Tab:RS1512}
\end{table}
The basis of the space of its local solutions around $\dot z=\frac{1}{2}$ in 
Lemma \ref{lem:basis} becomes 
\begin{align*}
&\begin{pmatrix}
\ds{\exp(\frac{11\pi\i}{12})\int_{-\infty}^0(-t)^{-7/12}(z-t)^{-5/12}(1-t)^{-1/12}dt}\\
\ds{\int^{\infty}_1(t)^{-7/12}(t-z)^{-5/12}(t-1)^{-1/12}dt}\\
\end{pmatrix}\\
=&
\begin{pmatrix}
\ds{\exp(\frac{11\pi\i}{12})
B(\frac{1}{12},\frac{5}{12})F(\frac{1}{12},\frac{5}{12},\frac{1}{2},1-z)}\\[2mm]
\ds{B(\frac{1}{12},\frac{11}{12})F(\frac{1}{12},\frac{5}{12},1,z)}
\end{pmatrix}.
\end{align*}
The circuit matrices with respect to this basis are 
$$M_0=\begin{pmatrix} 1 & 1-e^{-\pi\i/6}\\ 0 & 1\end{pmatrix},\quad 
M_1=\begin{pmatrix} 1 & 0 \\ e^{7\pi\i/6}-1 & -1\end{pmatrix}.
$$

As in the previous section, we can find a matrix 
$R=\begin{pmatrix} -\i-\w& -\i \\ 0 & 1\end{pmatrix}$, which satisfies 
$$N_0=RM_0R^{-1}=\begin{pmatrix} 1& 1 \\ 0 & 1\end{pmatrix}=T, \quad 
N_1=RM_0R^{-1}=\i \begin{pmatrix} 0& 1 \\-1 & 0\end{pmatrix}=\i J,$$
$$
N_\infty=(N_0N_1)^{-1}=\i \begin{pmatrix} 0& 1 \\ -1 & 1\end{pmatrix}.
$$
Since $N_\infty^3=\i I_2$ and $\SL_2(\Z)$ is generated by $T$ and $J$,   
the group generated by $N_0$ and $N_1$ is 
$$\SL_2(\Z)\la \i I_2\ra=\{\i^k g\in \GL_2(\Z[\i])\mid 
g\in \SL_2(\Z), k\in \Z\},$$
whose projectivization is isomorphic to $\PSL_2(\Z)$. 

\begin{definition}[Schwarz's map $\f_2$ for 
$(a,b,c)=(\frac{1}{12},\frac{5}{12},1)$]   
\label{def:Schwarz2}   
We define Schwarz's map $\f_2$ for $(a,b,c)=(\frac{1}{12},\frac{5}{12},1)$ 
by the analytic continuation of the map 
\begin{equation}
\label{eq:Schwarz2}
\f_2:z\mapsto \tau
=\tau(z)=\i\cdot \frac{
  {B} \left(\frac{1}{12},{\frac{5}{12}}\right)}
{2\pi}\cdot 
\frac{
{F(\frac{1}{12},{\frac{5}{12}},\frac{1}{2};\,1-z)}}
{F(\frac{1}{12},{\frac{5}{12}},1;\,z)}-\i.
\end{equation}
defined on $\B_0\cap \B_1$ to $\C-\{0,1\}$. 
Its projective monodromy representation is isomorphic to $\PSL_2(\Z)$.
\end{definition}

\begin{remark}
Schwarz's map for $(a,b,c)=(\frac{1}{12},{\frac{5}{12}},1)$ 
is studied in \cite[Proposition 6.4.1]{Sh} by a different basis 
of the space of local solutions from ours.
\end{remark}

\begin{proposition}
\label{prop:ST23inf}
Schwarz's triangle $\f_2(\H)$ for $(a,b,c)=(\frac{1}{12},\frac{5}{12},1)$ is 
$$\{\tau\in \H\mid 0 <\re(\tau) <\frac{1}{2},\ |\tau|>1\},$$
which is the interior of the right half of the fundamental region 
$\D$ for $\PSL_2(\Z)$, 
and its vertices $\f_2(z_0)=\lim\limits_{z\to z_0,z\in \H} \f_2(z)$
$(z_0=0,1,\infty)$ are 
\begin{equation}
\label{eq:f2-vertices}
\f_2(0)=\i\infty, \quad \f_2(1)=\i, \quad \f_2(\infty)=-\w^2.
\end{equation}
See Figure \ref{fig:FD-SL2Z} for their shape. 
\end{proposition}

\begin{proof}
Fact \ref{fact:G-K} together with $c-a-b=0$ for $(a,b,c)=
(\frac{1}{12},\frac{5}{12},\frac{1}{2})$ yields $\f_2(0)=\i\infty$. 
It also yields that
\begin{align*}\f_2(1)&=\lim_{z\to 1,z\in(0,1)}
\Big(\i\cdot \frac{
  {B} \left(\frac{1}{12},{\frac{5}{12}}\right)}
{2\pi}\cdot 
\frac{
{F(\frac{1}{12},{\frac{5}{12}};\,\frac{1}{2};\,1-z)}}
{F(\frac{1}{12},{\frac{5}{12}};\,1;\,z)}-\i\Big)\\
&=
\i\cdot \frac{
  \G\left(\frac{1}{12}\right)\G\left({\frac{5}{12}}\right)}
{2\pi\G\left(\frac{1}{2}\right)}\cdot 
\frac{
{\G\left(\frac{11}{12}\right)\G\left(\frac{7}{12}\right)}}
{\G(1)\G\left(\frac{1}{2}\right)}-\i= \frac{\i\cdot \pi^2}
{2\pi^2 \cdot \sin\left(\frac{\pi}{12}\right)
\sin\left(\frac{5\pi}{12}\right)}-\i
=\i.
\end{align*}
Note that the eigenvalues of $N_\infty$ are $e^{\pi\i/6}=\w/\i$ and 
$e^{5\pi\i/6}=\w^2/\i$ and that their eigen row vectors are $(1,\w)$ and 
$(1,\w^2)$, respectively.
We can show $\f_2(\infty)=-\w^2$ quite similarly to $\f_1(\infty)=-\w^2$ 
in Proof of Proposition \ref{prop:ST33inf}.
Since each edge of Schwarz's triangle $\f_2(\H)$ is 
parallel to the imaginary axis or in a circle with center in the real axis, 
and bounded by two of $\f_2(0)=\i\infty$, $\f_2(1)=\i$, $\f_2(\infty)=-\w^2$, 
the edges are contained in the lines $\re(\tau)=0$ or $\re(\tau)=\frac{1}{2}$, 
or the circle $|\tau|=1$. 
\end{proof}

\begin{remark}
The image of $Z=\C-\{0,1\}$ under $\f_2$ is 
$$\H_2=\H-\{g\cdot\i,\ g\cdot(-\w^2)\in \H\mid g\in \PSL_2(\Z)\}.$$
We can extend Schwarz's map $\f_2$ to a map from $\P^1-\{0\}$ to 
$\H$ by corresponding $1$ and $\infty$ to 
the vertices $\i=\f_2(1)$ and $-\w^2=\f_2(\infty)$ of Schwarz's triangles 
or their equivalent points under the action of $\PSL_2(\Z)$. 
Moreover, we can extend it to a map from $\P^1$ to 
$\wh{\H}=\H\cup\{\i\infty\}\cup \Q$. 
Note that any point in  $\Q$ is equivalent to $\i\infty=\f_2(0)$ 
under the action of $\PSL_2(\Z)$. 
These extensions of Schwarz's map $\f_2$ are denoted by 
the same symbol $\f_2$.
\end{remark}

\begin{theorem}
\label{th:inv-f2}
The inverse of Schwarz's map $\f_2$ is given by 
\begin{equation}
\label{eq:inv-f2}
\H_2\ni \tau \mapsto z=\frac{1}{j(\tau)}\in \C-\{0,1\}, 
\end{equation}
where the $j$-invariant is defined by
\begin{equation}
\label{eq:j-inv}
j(\tau)=\frac{4}{27}\frac{(\l(\tau)^2-\l(\tau)+1)^3}{\l(\tau)^2(1-\l(\tau))^2}
=
\frac{1}{54}
\frac{(\h_{00}(\tau)^8+\h_{01}(\tau)^8+\h_{10}(\tau)^8)^3}
{\h_{00}(\tau)^8\h_{01}(\tau)^8\h_{10}(\tau)^8}.
\end{equation}
\end{theorem}

\begin{remark}
We refer to \cite[p.42]{Yo} for our definition of $j(\tau)$.
As in \cite[(23) in p. 90]{Se}, $1728j(\tau)$ admits a Fourier expansion 
with integral coefficients: 
$$1728j(\tau)=\frac{1}{q}+744+196884q+\cdots,$$
where $q=\exp(2\pi\i\tau)$.
\end{remark}

\begin{proof} Schwarz's map $\f_2$ leads to an isomorphism between 
$Z$ and $\H_2/\PSL_2(\Z)$, which is extended to that 
between  $\P^1$ and $\wh{\H_2}/\PSL_2(\Z)$. 
Since $\PSL_2(\Z)/\PGa(2)$ is isomorphic to $S_3$, 
the space $\H_2/\PSL_2(\Z)$ is regarded as the quotient of $\H/\PGa(2)$ 
by the symmetric group $S_3$. 
To distinguish the space $\C-\{0,1\}$ isomorphic to $\H_2/\PSL_2(\Z)$ and 
that to $\H/\PGa(2)$, the former is denoted by $Z_{1/j}$ and 
the latter by $Z_\l$. The space $Z_\l$ is regarded as a 
$6$-fold covering of $Z_{1/j}$. Thus we have a diagram
$$\begin{array}{ccc} 
\H/\PGa(2) &\overset{\lambda}{\longrightarrow} &Z_\l
\\[3mm]
\downarrow & &\downarrow\frac{1}{\jmath}
\\[3mm]
\H/\PSL_2(\Z) &
\overset{\sim}{\longrightarrow}
&Z_{1/j}.
  \end{array}
  $$

Since the right hand side of \eqref{eq:j-inv} is invariant under 
the action of $SL_2(\Z)$ by Fact \ref{fact:SL2act}, 
the function $\dfrac{1}{j(\tau)}$ on $\H/\PGa(2)$ descends to  
a function on $\H/\PSL_2(\Z)$.   
We decompose the function $\dfrac{1}{j(\tau)}$ into 
the $\l$-function and a rational function 
$$\frac{1}{\jmath}:Z_\l \ni \l \mapsto \frac{27}{4}
\frac{\l^2(1-\l)^2}{(\l^2-\l+1)^3}\in Z_{1/j}$$
of degree $6$. 
This function is regarded as  
a branched covering map from $\P^1$ to $\P^1$ of degree $6$.
Since the natural projection from $\H/\PGa(2)$ to $\H/\PSL_2(\Z)$ 
is of degree $6$, it turns out that 
the function $\dfrac{1}{j(\tau)}$ on $\H/\PSL_2(\Z)$ is of degree $1$. 
We can easily see that 
the function $\dfrac{1}{\jmath}$ maps $\l=0,1,\infty$ to $0$ and 
$\l=-\w,-\w^2$ to $\infty$. 
By solving an equation 
$$\frac{27}{4}
\frac{\l^2(1-\l)^2}{(\l^2-\l+1)^3}=1,$$
we also see that $\dfrac{1}{\jmath}$ maps $\l=-1,\dfrac{1}{2},2$ to $1$.
Recall that the function $\l$ sends 
$\tau=\i\infty,\i,-\w^2$ to $0,\dfrac{1}{2},-\w^2$  
as in Fact \ref{fact:Inv-S-map}(1) and Lemma \ref{lem:midpt} (2). 
Since $\dfrac{1}{j(\tau)}$ is extended to 
a map $\wh{\H}/\PSL_2(\Z)\to \P^1$ of degree $1$ with sending 
$\tau=\i\infty, \i,-\w^2$ to $0, 1,\infty$, 
it coincides with the inverse of $\f_2$ by \eqref{eq:f2-vertices}.        
\end{proof}

\begin{theorem}
\label{th:Ana2-Jacobi}
We have 
\begin{equation}
\label{eq:Jacobi-2}
E_4(\tau)=\frac{\h_{00}(\tau)^8+\h_{01}(\tau)^8+\h_{10}(\tau)^8}{2}=
F(\frac{1}{12},\frac{5}{12},1;\frac{1}{j(\tau)})^4
\end{equation}
for any $\tau$ in the interior $\D^\circ$ of $\D$, where 
the $j$-invariant $j(\tau)$ is given in  \eqref{eq:j-inv}. 
It is  extended to an equality on the whole space $\H$. 
\end{theorem}
\begin{proof}
We can show this theorem quite similarly to Proof of Theorem 
\ref{th:Ana1-Jacobi}.
\end{proof}

\begin{remark}
\label{rem:multi-valued-2}
Though $F(\frac{1}{12},\frac{5}{12},1;\frac{1}{j(\tau)})^4$
is a single-valued holomorphic function on $\H$, 
neither $F(\frac{1}{12},\frac{5}{12},1;\frac{1}{j(\tau)})$ 
nor $F(\frac{1}{12},\frac{5}{12},1;\frac{1}{j(\tau)})^2$ is so. 
We can see that any element of the $\PSL_2(\Z)$-orbit 
$$\{g\cdot (-\w^2)\in \H\mid g\in \PSL_2(\Z)\}$$ 
of $-\w^2$ is a simple zero of 
$F(\frac{1}{12},\frac{5}{12},1;\nu(\tau))^4=E_4(\tau)$. 
Note also that the monodromy 
representation for $\cE(\frac{1}{12},\frac{5}{12},1)$ includes 
the scalar multiple of $\i$. 
Since their restrictions to $\{\tau\in \H\mid \im(\tau)>1\}$ 
are single valued and have a period $1$,  
each of them admits a Fourier expansion: 
\begin{align*}
F(\frac{1}{12},\frac{5}{12},1;\frac{1}{j(\tau)})
&=1+60e^{2\pi\i\tau}-4860(e^{2\pi\i\tau})^2+
660480(e^{2\pi\i\tau})^3
-\cdots,\\
F(\frac{1}{12},\frac{5}{12},1;\frac{1}{j(\tau)})^2
&=1+120e^{2\pi\i\tau}-6120(e^{2\pi\i\tau})^2+737760(e^{2\pi\i\tau})^3
-\cdots;    
\end{align*}
see \eqref{eq:Fourier}     
for the Fourier expansion of 
$F(\frac{1}{12},\frac{5}{12},1;\frac{1}{j(\tau)})^4=E_4(\tau)$.  
\end{remark}

By pulling back equalities in Theorem \ref{th:Ana2-Jacobi} under 
Schwarz's map $\f_2$ 
in Definition \ref{def:Schwarz2}, we have the following. 
\begin{cor}
\label{cor:Ana2-Jacobi}
We have 
\begin{equation}
\label{eq:Jacobi-2-z}
F(\frac{1}{12},\frac{5}{12},1;z)^{4}=
\frac{\h_{00}(\f_2(z))^8+\h_{01}(\f_2(z))^8+\h_{10}(\f_2(z))^8}{2}=E_4(\f_2(z)),
\end{equation}
which is initially an equality on 
$\B_0\cap\B_1$ 
and extended to that on the whole space $Z$ by the analytic continuation. 
\end{cor}

\section{Functional equations for $F(a,b,c;z)$}
\begin{cor} 
\label{cor:FE-HGS}
The hypergeometric series $F(a,b,c;z)$ satisfies 
functional equations:
\begin{align}
\nonumber
F(\frac{1}{6},\frac{1}{2},1;\frac{3\sqrt{3}\i z(1-z)}{(1+\w z)^3})^2
=&-\w^2F(\frac{1}{2},\frac{1}{2},1;z)^2
-\w F(\frac{1}{2},\frac{1}{2},1;\frac{z}{z-1})^2,\\
\nonumber
%
F(\frac{1}{6},\frac{1}{2},1;-\frac{3\sqrt{3}\i z(1-z)}{(z+\w)^3})^2
=&-\w F(\frac{1}{2},\frac{1}{2},1;z)^2
-\w^2F(\frac{1}{2},\frac{1}{2},1;\frac{z}{z-1})^2,\\
\label{eq:FE-HGS}
F(\frac{1}{12},\frac{5}{12},1;\frac{z^2}{4(z-1)})^2
=&F(\frac{1}{6},\frac{1}{2},1;z)F(\frac{1}{6},\frac{1}{2},1;\frac{z}{z-1}),\\
\nonumber
F(\frac{1}{12},\frac{5}{12},1;\frac{27z^2(1-z)^2}{4(z^2-z+1)^3})^4
=&F(\frac{1}{2},\frac{1}{2},1;z)^4
+F(\frac{1}{2},\frac{1}{2},1;\frac{z}{z-1})^4\\
\nonumber
&-F(\frac{1}{2},\frac{1}{2},1;z)^2F(\frac{1}{2},\frac{1}{2},1;\frac{z}{z-1})^2,
\end{align}
where $z$ is in a small neighborhood of $z=0$. 
\end{cor}

\begin{proof}
We set 
$z=\l(\tau)=\dfrac{\h_{10}(\tau)^4}{\h_{00}(\tau)^4}.$ 
Then Theorem \ref{th:Ana1-Jacobi} and Fact \ref{fact:Jacobi}
yield that 
\begin{align*}
F(\frac{1}{6},\frac{1}{2},1;\frac{3\sqrt{3}\i z(1-z)}{(1+\w z)^3})^2
&=\h_{00}(\tau)^4+\w\h_{10}(\tau)^4
=-\w^2\h_{00}(\tau)^4-\w\h_{01}(\tau)^4
\\
&
=-\w^2F(\frac{1}{2},\frac{1}{2},1;z)^2
-\w F(\frac{1}{2},\frac{1}{2},1;\frac{z}{z-1})^2.
\end{align*}

We obtain the second equality by acting $T\in SL_2(\Z)$ on $\tau$ 
in the first equality as in Proof of Corollary \ref{cor:Ana1-Jacobi}. 

Theorem \ref{th:Ana2-Jacobi}, Fact \ref{fact:Jacobi}, 
Remark \ref{rem:rel-theta's} and the equality 
$j(\tau)=\dfrac{4}{27}\cdot 
\dfrac{(\l(\tau)^2-\l(\tau)+1)^3}{\l(\tau)^2(1-\l(\tau))^2}$ 
yield the last equality. 

To obtain the third equality, we set 
$z=\nu(\tau)=\dfrac{3\sqrt{3}\i \l(\tau)(1-\l(\tau))}{(1+\w \l(\tau))^3}$. 
Then we have 
$$F(\frac{1}{6},\frac{1}{2},1;z)F(\frac{1}{6},\frac{1}{2},1;\frac{z}{z-1})
=\frac{\h_{00}(\tau)^8+\h_{01}(\tau)^8+\h_{10}(\tau)^8}{2}\\
=F(\frac{1}{12},\frac{5}{12},1;\frac{z^2}{4(z-1)})^4
$$
by Corollary \ref{cor:8jyouwa}, Theorem \ref{th:Ana2-Jacobi} and the 
equality $j(\tau)=4\cdot \dfrac{\nu(\tau)-1}{\nu(\tau)^2}$.  
\end{proof}





\end{document}